\documentclass[preprint,3p]{elsarticle}
\usepackage[title]{appendix}

\usepackage{stmaryrd}

\SetSymbolFont{stmry}{bold}{U}{stmry}{m}{n}
\usepackage{amsmath}
\usepackage{amssymb}
\usepackage{amsthm}
\usepackage{tabularx}
\usepackage{subfigure}
\usepackage{epsfig}
\usepackage{indentfirst}
\usepackage{cases}
\biboptions{sort&compress}
\usepackage{graphicx}
\usepackage{epsfig}
\usepackage{color}
\usepackage{float}
\usepackage{multirow}

\renewcommand{\vec}[1]{\mathbf{#1}}

\def\subFV{\scriptscriptstyle{FV}}
\def\subFS{\scriptscriptstyle{FS}}
\def\subVS{\scriptscriptstyle{VS}}

\newcommand{\be}{\begin{equation}}
\newcommand{\ee}{\end{equation}}
\newcommand{\ba}{\begin{array}}
\newcommand{\ea}{\end{array}}
\newcommand{\bea}{\begin{eqnarray}}
\newcommand{\eea}{\end{eqnarray}}
\newcommand{\beas}{\begin{eqnarray*}}
\newcommand{\eeas}{\end{eqnarray*}}

\newtheorem{thm}{Theorem}[section]
\newtheorem{prop}{Proposition}[section]
\newtheorem{remark}{Remark}[section]
\newtheorem{coro}{Corollary}[section]
\newtheorem{lem}{Lemma}[section]

\numberwithin{equation}{section}
\renewcommand{\theequation}{\arabic{section}.\arabic{equation}}
\newcommand{\fakesection}[1]{%
  \par\refstepcounter{section}
  \sectionmark{#1}
  \addcontentsline{toc}{section}{\protect\numberline{\thesection}#1}
}

\begin{document}

\begin{frontmatter}

\title{An energy-stable parametric finite element method for\\
 anisotropic surface diffusion}

\author[1]{Yifei Li}
\address[1]{Department of Mathematics, National University of
Singapore, Singapore, 119076}
\ead{e0444158@u.nus.edu}
\author[1]{Weizhu Bao\corref{2}}
\ead{matbaowz@nus.edu.sg}
\cortext[2]{Corresponding author.}


\begin{abstract}

We propose an energy-stable parametric finite element method (ES-PFEM)
to discretize the motion of a closed curve under surface diffusion with an anisotropic surface energy $\gamma(\theta)$ -- anisotropic surface diffusion -- in two dimensions, while $\theta$ is the angle between the outward unit normal vector and the vertical axis. By introducing a positive definite surface energy (density) matrix $G(\theta)$, we present a new and simple variational formulation for the anisotropic surface diffusion and prove that it satisfies area/mass conservation and energy dissipation. The variational
problem is discretized in space by the parametric finite element method
and area/mass conservation and energy dissipation are established for the
semi-discretization. Then the problem is further discretized in time
by a (semi-implicit) backward Euler method so that only a linear system is to be solved at each time step for the full-discretization and thus it is efficient. We establish
well-posedness of the full-discretization and identify some simple conditions on $\gamma(\theta)$ such that the full-discretization keeps energy dissipation and thus it is unconditionally energy-stable.
Finally the ES-PFEM is applied to simulate solid-state dewetting of thin films with anisotropic surface energies, i.e. the motion of an open curve under anisotropic surface diffusion with proper boundary conditions at
the two triple points moving along the horizontal substrate. Numerical results are reported to demonstrate
the efficiency and accuracy as well as energy dissipation of the proposed
ES-PFEM.
\end{abstract}




\end{frontmatter}


\section{Introduction}
Surface diffusion is a general and important process involving the motion of adatoms, atomic clusters (adparticles), and molecules at material surfaces and interfaces in solids \cite{Oura}. It is an important mechanism and/or kinetics in epitaxial growth, surface phase formation, heterogeneous catalysis, and other areas in surface/materials  science \cite{Shu}. Due to different surface lattice orientations at material surface in solids, orientational anisotropy is a general pattern
in both diffusion rates and mechanisms at the various surface orientations of a given material. This orientational anisotropy causes anisotropic surface
energy and thus generates {\bf anisotropic surface diffusion} at material surfaces and interfaces in solids \cite{Oura,Thompson12}. In fact,
surface/anisotropic surface diffusion has manifested broader and significant applications in materials science and solid-state physics as well as computational geometry, such as crystal growth of nanomaterials \citep{cahn1991stability}, morphology development in alloys, evolution of voids in microelectronic circuits \citep{li1999numerical}, solid-state dewetting \cite{Thompson12,Ye10a,Srolovitz86,Jiang2012,wang2015sharp}, deformation of images \citep{clarenz2000anisotropic}, etc.

The mathematical model for surface diffusion in materials science can be
traced back to the work by Mullins \citep{mullins1957theory} for describing the diffusion at interfaces in alloys. Later, Dav{\`i} and Gurtin \citep{davi1990motion} extended the model to anisotropic surface diffusion.
By introducing the weighted mean curvature, Cahn and Taylor \citep{Cahn94,taylor1994linking} proposed a simple mathematical model and
showed that it is equivalent to the model in the literature for the anisotropic surface diffusion. For more details, we refer \cite{cahn1991stability,Bao17,jiang2016solid,Jiang} and
references therein.


As illustrated in Figure \ref{fig:model2d}, let $\Gamma:= \Gamma (t)$ be
a closed curve in two dimensions (2D), which is represented by $\vec X:=\vec X(s,t)=(x(s,t),y(s,t))^T\in{\mathbb R}^2$ with $t$ denoting the time and $s$ being the arc length parametrization of $\Gamma$. The motion of $\Gamma$ under anisotropic surface diffusion is governed by the following geometric partial differential equation (PDE) \cite{Cahn94,Bao17,jiang2016solid}:
\begin{equation}\label{anisotropic surface diffusion, original}
	\partial_t\vec X=\partial_{ss}\mu\,\vec n,
\end{equation}
where $\vec \tau=(\cos\theta,\sin\theta)^T$ is the unit tangent vector, $\vec n=(-\sin \theta,\cos \theta)^T$  is the outward unit normal vector with $\theta$ being the angle between $\vec n$ and the vertical axis,  and $\mu:=\mu(s,t)$ is the weighted mean curvature (or chemical potential) defined as \cite{Cahn94,Bao17,jiang2016solid}:
\begin{equation}\label{def of mu, original}
	\mu=\left[\gamma(\theta)+\gamma''(\theta)\right]\kappa,
\end{equation}
with $\kappa:=-(\partial_{ss}\vec X)\cdot \vec n$ being the curvature and $\gamma(\theta)\in C^2([-\pi,\pi])$ being the surface energy, which is a
dimensionless positive and periodic function satisfying $\gamma(-\pi)=\gamma(\pi)$ and
$\gamma^\prime(-\pi)=\gamma^\prime(\pi)$. The initial data for
\eqref{anisotropic surface diffusion, original} is given as
\begin{equation}\label{init}
\vec X(s,0)=\vec X_0(s)=(x_0(s),y_0(s))^T, \qquad 0\le s\le L_0,
\end{equation}
where $L_0$ is the length of the initial curve $\Gamma_0=\Gamma(0)$.
\begin{figure}[http]
\centering
\includegraphics[width=0.6\textwidth]{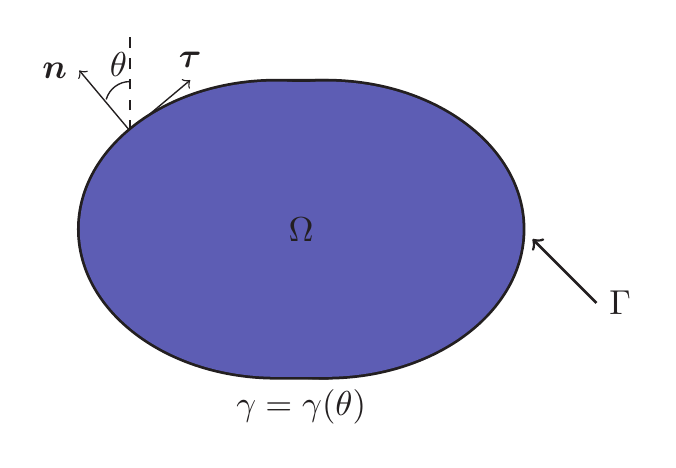}
\caption{An illustration of a closed curve $\Gamma$ in two dimensions under anisotropic surface diffusion with an anisotropic surface energy $\gamma(\theta)$, while $\theta$ is the angle between the outward unit normal vector $\vec n$ and the $y$-axis.}
\label{fig:model2d}
\end{figure}

Since $\Gamma(t)$ is parameterized by the arc length parameter $s$, the tangential vector $\boldsymbol\tau$ and the outward unit normal vector ${\bf n}$ can be expressed as
\begin{equation}\label{normtang}
{\boldsymbol \tau}=\partial_{s}\vec X=(\partial_s x,\partial_s y)^T, \quad
{\bf n}=-{\boldsymbol \tau}^\perp=-\partial_{s}\vec X^\perp=
(-\partial_s y,\partial_s x)^T, \quad \partial_s x=\cos\theta, \quad \partial_s y=\sin\theta.
\end{equation}
In addition, the curvature $\kappa$ can also be formulated by $s$ and $\theta$ as
\begin{equation}\label{def of kappa by theta}
\kappa=-(\partial_{ss}\vec X)\cdot \vec n= \partial_{ss}x\,\partial_s y-\partial_{ss}y\,\partial_s x =-(\sin^2\theta+\cos^2\theta)\partial_s\theta =-\partial_s\theta.
\end{equation}

When $\gamma(\theta)\equiv 1$ for $\theta\in [-\pi,\pi]$, it is called as {\sl isotropic} surface energy; and in this case, $\mu=\kappa$
in \eqref{def of mu, original}, and
\eqref{anisotropic surface diffusion, original} is for surface diffusion
\cite{barrett2007parametric,mullins1957theory,Jiang2020,BaoZ2021}.
On the contrary, when $\gamma(\theta)$ is not a constant function,
it is called as {\sl anisotropic} surface energy; and in this case, $\mu$ is
called as the weighted mean curvature (or chemical potential), and
\eqref{anisotropic surface diffusion, original} is for anisotropic surface diffusion \cite{Jiang,taylor1994linking}. In addition, when $\tilde \gamma(\theta): = \gamma(\theta)+\gamma''(\theta)>0$ for $\theta\in [-\pi,\pi]$, i.e.
the surface stiffness $\tilde \gamma(\theta)$ does not change sign,
it is called as {\sl weakly anisotropic}; and when $\tilde \gamma(\theta)$ changes sign for $\theta\in [-\pi,\pi]$, it is called as {\sl strongly anisotropic}. In this paper, we assume that $\gamma(\theta)$ is isotropic/weakly anisotropic, i.e. $\gamma(\theta)+\gamma''(\theta)>0$ for $\theta\in [-\pi,\pi]$. Typical anisotropic surface energy $\gamma(\theta)$ in materials science includes:
\medskip

(i)  the $k$-fold anisotropy surface energy \cite{bao2017parametric}
\begin{equation}\label{kfold}
	\gamma(\theta)=1+\beta\cos(k(\theta-\theta_0)),\qquad \theta\in[-\pi,\pi],
\end{equation}
where $k=2,3,4,6$, $\beta$ is the dimensionless anisotropic strength constant and $\theta_0\in [-\pi,\pi]$ is a constant;

(ii) the ellipsoidal anisotropy surface energy \cite{taylor1994linking}
\begin{equation}\label{ellipsoidal}
	\gamma(\theta)=\sqrt{a+b\cos^2\theta},\qquad  \theta\in[-\pi,\pi],
\end{equation}
where $a$ and $b$ are two dimensionless constants satisfying $a>0$ and $a+b>0$; and

(iii) the Riemannian  metric  anisotropy surface energy \cite{barrett2008variational}
\begin{equation}\label{BGN}
	\gamma(\theta)=\sum_{k=1}^K\sqrt{\vec n(\theta)^T G_k\vec n(\theta)},\qquad \hbox{with}\quad  \vec n(\theta)=(-\sin \theta,\cos \theta)^T, \qquad \theta\in[-\pi,\pi],
\end{equation}
where $K$ is a positive integer, and $G_k \in \mathbb{R}^{2\times 2} (k=1,2,\ldots,K)$ are symmetric positive definite matrices. We remark here
that when $K=1$ and $G_1={\rm diag}(a, b+a)$ in \eqref{BGN}, then the Riemannian  metric  anisotropy surface energy \eqref{BGN} collapses
to the ellipsoidal anisotropy surface energy \eqref{ellipsoidal}.


Let $A(t)$ be the area/mass of the film  (i.e., the  region $\Omega(t)$ enclosed by the curve $\Gamma(t)$) and $W_c(t)$ be the total interfacial free energy,
which are defined as
\begin{equation}\label{AtWct}
A(t):=\int_{\Omega(t)}1\,d{\bf x}=\int_0^{L(t)}y(s,t)\partial_sx(s,t)\,ds, \qquad W_c(t):=\int_{\Gamma(t)}\gamma(\theta)\,ds=\int_0^{L(t)}\gamma(\theta)\,ds, \qquad t\ge0,
\end{equation}
where $L(t):=\int_{\Gamma(t)}1\,ds$ is the length of $\Gamma(t)$, one can prove that \cite{bao2017parametric,Bao17,Barrett07b}
\begin{equation}
\label{eqn:area conservation pde}
\frac{d}{dt}A(t)=0,\qquad
\frac{d}{dt}W_c(t)=-\int_0^{L(t)}(\partial_s \mu)^2ds\leq 0, \qquad t\ge0,
\end{equation}
which immediately implies the anisotropic surface diffusion \eqref{anisotropic surface diffusion, original}-\eqref{def of mu, original}
with \eqref{init} satisfies area/mass conservation and
energy dissipation, i.e.
\begin{equation}
A(t) \equiv A(0)=\int_0^{L_0}y_0(s)x_0^\prime(s)\,ds,\qquad W_c(t)\le W_c(t_1)\le W_c(0)=\int_{\Gamma_0}\gamma(\theta)\,ds, \qquad t\ge t_1\ge0.
\label{eqn:massenergy2dd}
\end{equation}

For the surface diffusion, i.e. $\gamma(\theta)\equiv 1$ in
\eqref{def of mu, original}, by reformulating
\eqref{anisotropic surface diffusion, original}-\eqref{def of mu, original}
with $\gamma(\theta)\equiv 1$ into
\begin{subequations}
\label{eqn:2disonew}
\begin{numcases}{}
\label{eqn:2disonew1}
\vec{n}\cdot \partial_{t}\vec{X}-\partial_{ss}\kappa=0,  \\
\label{eqn:2disonew2}
\kappa\,\vec{n}+\partial_{ss}\vec{X}=0, \qquad 0<s<L(t),\quad t>0,
\end{numcases}
\end{subequations}
Barrett {\it et al.} \cite{barrett2007parametric,Barrett07b,barrett2019finite} introduced a novel variational formulation of \eqref{eqn:2disonew} and presented an elegant parametric finite element method (PFEM) for the evolution of a closed curve under surface diffusion. The PFEM has a few good properties including unconditional stability, energy dissipation and asymptotic mesh equal distribution (AMED). The proposed PFEM was successfully extended for simulating the anisotropic surface diffusion with the specific Riemannian  metric  anisotropy surface energy \eqref{BGN} by adapting a variational formation of \eqref{anisotropic surface diffusion, original}-\eqref{def of mu, original} via the anisotropic surface energy $\gamma$ in terms of  $\gamma(\vec{n})$ instead of $\gamma(\theta)$ by Barrett {\it et al.} \cite{barrett2008variational}.
The PFEM was also extended for solving the anisotropic surface diffusion with applications in simulating solid-state dewetting by reformulating
\eqref{anisotropic surface diffusion, original}-\eqref{def of mu, original}
into
\begin{subequations}
\label{eqn:2disonews}
\begin{numcases}{}
\label{eqn:2disonews1}
\vec{n}\cdot \partial_{t}\vec{X}-\partial_{ss}\mu=0,  \\
\label{eqn:2disonews2}
\mu=\left[\gamma(\theta)+\gamma''(\theta)\right]\kappa,\qquad 0<s<L(t),\quad t>0,\\
\label{eqn:2disonews3}
\kappa\,\vec{n}+\partial_{ss}\vec{X}=0,
\end{numcases}
\end{subequations}
and obtaining a variational formulation with $(\vec{X},\mu,\kappa)$ as
unknown functions \cite{bao2017parametric}.  Unfortunately those good properties of
the PFEM for surface diffusion, such as unconditional stability, energy dissipation and asymptotic mesh equal distribution,  are lost in the above extension for general anisotropic surface diffusion \cite{bao2017parametric}.

The main aim of this paper is to present a new and simple variational formulation for the anisotropic surface diffusion
\eqref{anisotropic surface diffusion, original}-\eqref{def of mu, original}
with $(\vec{X},\mu)$ as unknown functions by introducing an anisotropic surface energy matrix $G(\theta)$ depending on $\gamma(\theta)$.
An energy-stable parametric finite element method (ES-PFEM)
is then proposed for the discretization of the new variational problem
under some simple conditions on $\gamma(\theta)$. The proposed ES-PFEM
for anisotropic surface diffusion enjoys most good properties of the original
PFEM for surface diffusion, such as semi-implicit and thus efficient,
unconditional stability, energy dissipation and asymptotic mesh quasi-equal distribution. The proposed ES-PFEM is extended to simulate solid-state dewetting, i.e. the motion of an open curve under anisotropic surface diffusion and contact line migration \cite{bao2020energy}.

The rest of the paper is organized as follows: In section 2, we present a new and simple  variational formulation and prove its area/mass conservation and energy dissipation. In section 3, we propose a semi-discretization in space by PFEM for the variational problem and show its area/mass conservation and energy dissipation. In section 4, we present a full-discretization by adapting a (semi-implicit) backward Euler method in time, establish
well-posedness of the full-discretization and identify some simple conditions on $\gamma(\theta)$ such that the full-discretization is energy dissipative.
Extension of the ES-PFEM to simulate solid-state dewetting of thin films under anisotropic surface diffusion and contact line migration is presented
in section 5. Numerical results are reported in section 6 to demonstrate
the efficiency, accuracy and unconditional energy stability of the proposed ES-PFEM.
Finally, some conclusions are drawn in section 7.

\section{A new variational formualtion and its properties}

In this section, we present a new and simple variational formulation for
the anisotropic surface diffusion
\eqref{anisotropic surface diffusion, original}-\eqref{def of mu, original}
and establish its area/mass conservation and energy dissipation.

\subsection{The new formulation}

Similar to \eqref{eqn:2disonew} for the surface diffusion, we reformulate the
anisotropic surface diffusion
\eqref{anisotropic surface diffusion, original}-\eqref{def of mu, original}
for the evolution of a closed curve as
\begin{subequations}
\label{eqn:original formulation gamma(theta)}
\begin{numcases}{}
\label{eqn:original gamma(theta) aniso eq1}
\vec n\cdot \partial_t\vec X -\partial_{ss}\mu=0,\qquad 0<s<L(t),\qquad t>0,   \\
\label{eqn:original gamma(theta) aniso eq2}
\mu\,\vec n+\left[\gamma(\theta)+\gamma''(\theta)\right]\partial_{ss}\vec X=0.
\end{numcases}
\end{subequations}

In order to obtain a variational formulation of
\eqref{eqn:original formulation gamma(theta)},
for convenience, we introduce a time independent variable $\rho$ such that $\Gamma(t)$ can be parameterized over the fixed domain $\rho\in\mathbb{I}=[0,1]$ (here $\rho$ and $s$ can be respectively regarded as the Lagrangian and Eulerian variables of the closed curve $\Gamma(t)$, and
we do not distinguish $\vec X(\rho,t)$ and $\vec X(s,t)$ for representing $\Gamma(t)$ when there is no misunderstanding) as
\begin{equation}
\Gamma(t):=\vec X(\rho,t)=(x(\rho,t),~y(\rho,t))^T:\; \mathbb{I}\times [0, T]\;\rightarrow \;\mathbb{R}^2.
\end{equation}
Based on this parametrization, the arc length parameter $s$ can be given as $s(\rho,t)=\int_0^\rho |\partial_q\vec{X}|\,dq$, and we have $\partial_\rho s=|\partial_\rho\vec{X}|,\, ds=\partial_\rho s d\rho=|\partial_\rho \vec X|d\rho$.
We also introduce the functional space with respect to the
evolution of the closed curve $\Gamma(t)$ as
\begin{equation}
L^2(\mathbb{I})=\left\{u: \mathbb{I}\rightarrow \mathbb{R} \ |\   \int_{\Gamma(t)}|u(s)|^2 ds
=\int_{\mathbb{I}} |u(s(\rho,t))|^2 \partial_\rho s\, d\rho <+\infty \right\},
\end{equation}
equipped with the $L^2$-inner product
\be\label{inner product torus}
\big(u,v\big)_{\Gamma(t)}:=\int_{\Gamma(t)}u(s)\,v(s)\,ds=
\int_{\mathbb{I}}u(s(\rho,t))\, v(s(\rho,t)) \partial_\rho s\,d\rho,\qquad \forall\;u,v\in L^2(\mathbb{I}).
\ee
Extension of \eqref{inner product torus} to $L^2(\mathbb{I})^2$  is straightforward.
Moreover, define the Sobolev spaces
\begin{eqnarray*}
&&\mathbb{K}:=H^1(\mathbb{I})=\left\{u:\mathbb{I}\rightarrow \mathbb{R} \ | \ u\in L^2(\mathbb{I}), \ \partial_\rho u\in L^2(\mathbb{I})\right\},\\
&&\mathbb{K}_p:=H_p^1(\mathbb{I})
=\{u\in H^1(\mathbb{I})\ |\  u(0)=u(1)\},\qquad
\mathbb{X}_p:=H_p^1(\mathbb{I})\times H_p^1(\mathbb{I}).
\end{eqnarray*}
In addition,
for a vector ${\bf v}=(v_1,v_2)^T\in {\mathbb R}^2$, we denote ${\bf v}^\perp\in {\mathbb R}^2$ as its perpendicular vector (rotation clockwise by
$\pi/2$) defined as
\begin{equation}\label{defperp}
{\bf v}^\perp:=(v_2,-v_1)^T=-J\,{\bf v},  \qquad \hbox{with}\quad
J=\left(\begin{array}{cc}
0 &-1\\
1 &0\\
\end{array}\right),
\end{equation}
which immediately implies that
\begin{equation}\label{perp2}
({\bf v}^\perp)^\perp=-J\,{\bf v}^\perp=J^2{\bf v}=-{\bf v},
 \qquad {\bf v}=(v_1,v_2)^T\in {\mathbb R}^2.
\end{equation}

Multiplying a test function $\varphi(\rho)\in \mathbb{K}_p$ to \eqref{eqn:original gamma(theta) aniso eq1} and then integrating over $\Gamma(t)$, integrating by parts,  noting $\partial_s\mu(0,t)=\partial_s\mu(1,t)$ and $\varphi(0)=\varphi(1)$,
we have
\begin{eqnarray}\label{varf1 continuous torus}
\Bigl(\partial_{t}\vec{X},\varphi\vec{n}
\Bigr)_{\Gamma(t)}&=&\Bigl(\vec{n}\cdot\partial_{t}\vec{X},
\varphi\Bigr)_{\Gamma(t)}=
\Bigl(\partial_{ss}\mu,\varphi\Bigr)_{\Gamma(t)}\nonumber\\
&=&-\Bigl(\partial_{s}\mu,
\partial_s\varphi\Bigr)_{\Gamma(t)}+\left.\left(\varphi\partial_{s}\mu\right)
\right|_{\rho=0}^{\rho=1}\nonumber\\
&=&-\Bigl(\partial_{s}\mu,
\partial_s\varphi\Bigr)_{\Gamma(t)}.
\end{eqnarray}

To get the variational formulation of \eqref{eqn:original gamma(theta) aniso eq2}, noticing $\kappa=-\partial_s\theta$ in \eqref{def of kappa by theta}, we have
\begin{equation}\label{dgammads}
\partial_s\gamma(\theta)=\gamma^\prime(\theta)\,\partial_s\theta=-\kappa\, \gamma^{\prime}(\theta),\qquad \partial_s\gamma^\prime(\theta)=
\gamma^{\prime\prime}(\theta)\, \partial_s\theta =-\kappa\, \gamma''(\theta), \qquad \theta\in[-\pi,\pi].
\end{equation}
Combining \eqref{normtang} and \eqref{perp2} with ${\bf v}={\boldsymbol \tau}$, and noticing
$\kappa\, \vec n=-\partial_{ss}\vec X$, we obtain
\begin{equation}\label{a temp relationship}
\kappa\,\partial_s \vec X=\kappa\,{\boldsymbol \tau}=-\kappa\,({\boldsymbol \tau}^\perp)^\perp=\kappa\,(-{\boldsymbol \tau}^\perp)^\perp=\kappa{\boldsymbol n}^\perp= -\partial_{ss} \vec X^{\perp},\quad \kappa \partial_s\vec X^{\perp}=-\kappa{\boldsymbol n}= \partial_{ss}\vec X.
\end{equation}
Plugging \eqref{dgammads} into \eqref{eqn:original gamma(theta) aniso eq2},
noting \eqref{a temp relationship}, we get
\begin{align}\label{xi vector gamma(theta) ver}
\mu\vec{n}=&-\left[\gamma(\theta)+\gamma''(\theta)\right]\partial_{ss}\vec X\nonumber\\
=&\partial_s(-\gamma(\theta)\partial_s\vec X)+\partial_s\gamma(\theta)\partial_s\vec X-\gamma''(\theta)\kappa \partial_s\vec X^{\perp}\nonumber\\
=&\partial_s(-\gamma(\theta)\partial_s\vec X)-\kappa\gamma'(\theta) \partial_s\vec X+\gamma''(\theta)\partial_s\theta \partial_{s}\vec X^{\perp}\nonumber\\
=&\partial_s(-\gamma(\theta)\partial_s\vec X)+\gamma'(\theta) \partial_{ss}\vec X^{\perp}+\partial_s\gamma'(\theta) \partial_{s}\vec X^{\perp}\nonumber\\
=&\partial_s\left(-\gamma(\theta) \partial_s \vec X+\gamma'(\theta) \partial_s \vec X^{\perp} \right).
\end{align}
Introducing the surface energy (density) matrix $G(\theta)$ as
\begin{equation}\label{gthta1}
	G(\theta)=\begin{pmatrix}
\gamma(\theta) & -\gamma'(\theta) \\
\gamma'(\theta) & \gamma(\theta)
\end{pmatrix},
\end{equation}
and noting \eqref{defperp} with ${\bf v}=\partial_s \vec X$, we have
\begin{equation}\label{usage of G}
-\gamma(\theta) \partial_s \vec X+\gamma'(\theta) \partial_s \vec X^{\perp}=-\left[\gamma(\theta) \partial_s \vec X-\gamma'(\theta) \partial_s \vec X^{\perp}\right]=-\left[\gamma(\theta)I_2 +\gamma'(\theta) J \right]\partial_s \vec X=-G(\theta)\partial_s \vec X,
\end{equation}
where $I_2$ is the $2\times 2$ identity matrix.
Substituting \eqref{usage of G} into \eqref{xi vector gamma(theta) ver}, we
obtain
\begin{equation}\label{munnew}
\mu\vec{n}=-\partial_s\left(G(\theta)\partial_s \vec X\right).
\end{equation}
Thus \eqref{eqn:original formulation gamma(theta)} (or
\eqref{anisotropic surface diffusion, original}-\eqref{def of mu, original}) is equivalent to the following conservative form:
\begin{subequations}
\label{eqn:original formulation gamma(theta)11}
\begin{numcases}{}
\label{eqn:original gamma(theta) aniso eq3}
\vec n\cdot \partial_t\vec X -\partial_{ss}\mu=0,\qquad 0<s<L(t),\qquad t>0,   \\
\label{eqn:original gamma(theta) aniso eq4}
\mu\vec{n}+\partial_s\left(G(\theta)\partial_s \vec X\right)=0.
\end{numcases}
\end{subequations}
Multiplying a test function $\boldsymbol{\omega}=(\omega_1,\omega_2)^T\in \mathbb{X}_p$ to \eqref{eqn:original gamma(theta) aniso eq2} and
then integrating over  $\Gamma(t)$, noticing \eqref{munnew} and
integrating by parts, noting $\theta(0,t)=\theta(1,t)$,
$G(\theta(0,t))=G(\theta(1,t))$, $\partial_s \vec X(0,t)=\partial_s \vec X(1,t)$ and $\boldsymbol{\omega}(0)=\boldsymbol{\omega}(1)$,
we get
\begin{eqnarray}\label{varf2 continuous torus}
\Bigl(\mu, \vec n\cdot\boldsymbol{\omega}\Bigr)_{\Gamma(t)}&=&\Bigl(\mu \vec n, \boldsymbol{\omega}\Bigr)_{\Gamma(t)}
=\Bigl(-\partial_s\left(G(\theta)\partial_s \vec X\right), \boldsymbol{\omega}\Bigr)_{\Gamma(t)}\nonumber\\
&=&\Bigl( G(\theta)\partial_s \vec X,\partial_s\boldsymbol{\omega}\Bigr)_{\Gamma(t)}+\left((- G(\theta)\partial_s \vec X)\cdot  \boldsymbol{\omega} \right)|_{\rho=0}^{\rho=1}\nonumber\\
&=&\Bigl( G(\theta)\partial_s \vec X,\partial_s\boldsymbol{\omega}\Bigr)_{\Gamma(t)}.
\end{eqnarray}

Combining \eqref{varf1 continuous torus} and \eqref{varf2 continuous torus}, we get a new and simple variational formulation for the anisotropic surface diffusion
\eqref{anisotropic surface diffusion, original}-\eqref{def of mu, original}
with the initial condition \eqref{init} as:  Given the initial curve $\Gamma(0):=\vec X(\rho,0)=\vec X_0(L_0\,\rho)\in \mathbb{X}_p$,  find the solution $\Gamma(t):=\vec X(\cdot, t)\in \mathbb{X}_p$ and $\mu(t)\in \mathbb{K}_p$ such that:
\begin{subequations}
\label{eqn:weak continuous torus}
\begin{align}
\label{eqn:weak1 continuous torus}
&\Bigl(\partial_{t}\vec{X},\varphi\vec{n}\Bigr)_{\Gamma(t)}+\Bigl(\partial_{s}\mu,
\partial_s\varphi\Bigr)_{\Gamma(t)}=0,\qquad\forall \varphi\in \mathbb{K}_p,\\[0.5em]
\label{eqn:weak2 continuous torus}
&\Bigl(\mu,\vec{n}\cdot\boldsymbol{\omega}\Bigr)_{\Gamma(t)}-\Bigl( G(\theta)\partial_s \vec X,\partial_s\boldsymbol{\omega}\Bigr)_{\Gamma(t)}= 0,\quad\forall\boldsymbol{\omega}\in\mathbb{X}_p.
\end{align}
\end{subequations}

\subsection{Area/mass conservation and energy dissipation}

Assume that the anisotropic surface energy $\gamma(\theta)\in C^1([-\pi,\pi])$ and $\gamma(-\pi)=\gamma(\pi)$, for the variational problem
\eqref{eqn:weak continuous torus}, we have

\begin{prop}[area/mass conservation and energy dissipation]
Let $\Bigl(\vec X(\cdot,~t),~\mu(\cdot,~t)\Bigr)\in \mathbb{X}_p\times \mathbb{K}_p $ be a solution of the variational problem
\eqref{eqn:weak continuous torus}. Then the area/mass $A(t)$ defined in \eqref{AtWct} is conserved and the total interfacial energy $W_c(t)$ defined in \eqref{AtWct} is dissipative, i.e. \eqref{eqn:massenergy2dd} is valid.
\end{prop}

\begin{proof} Differentiating $A(t)$ defined in \eqref{AtWct} with respect to $t$, integrating by parts, we get
\begin{eqnarray}\label{dAdt continuous}
\frac{d}{dt}A(t)&=&\frac{d}{dt}\int_0^{L( t)}y( s,t)\partial_s x(s,t)\;ds\nonumber\\
& = &\frac{d}{dt}\int_0^1 y(\rho,t)\partial_\rho x(\rho,t)\;d\rho = \int_0^1 (\partial_t y \partial_\rho x + y \partial_t\partial_\rho x)\;d\rho \nonumber\\
&=& \int_0^1 (\partial_t y \partial_\rho x - \partial_\rho y\partial_t x)d\rho +(y\partial_t x)\Big|_{\rho = 0}^{\rho = 1}\nonumber\\
&=& \int_{\Gamma(t)}(\partial_t\vec X)\cdot\vec n\;ds=\Bigl(\partial_{t}\vec{X},\vec{n}\Bigr)_{\Gamma(t)}, \qquad t\ge0.
\end{eqnarray}
Taking $\varphi\equiv 1$ in \eqref{eqn:weak1 continuous torus}, we have
\begin{equation}\label{phi=1 continuous}	\Bigl(\partial_{t}\vec{X},\vec{n}\Bigr)_{\Gamma(t)}=-\Bigl(\partial_{s}\mu,
\partial_s1\Bigr)_{\Gamma(t)}=-\Bigl(\partial_{s}\mu,
0\Bigr)_{\Gamma(t)}=0,\qquad t\ge0.
\end{equation}
Inserting \eqref{phi=1 continuous} into \eqref{dAdt continuous}, we obtain
\begin{equation}
\frac{d}{dt}A(t)=0, \qquad t\ge0,
\end{equation}
which immediately implies the area/mass conservation in \eqref{eqn:massenergy2dd}.

Similar to \eqref{def of kappa by theta}, we obtain
\begin{equation}\label{dWdt continuous pre1}
\partial_t\theta=(\sin^2\theta+\cos^2\theta)\partial_t\theta
=-\partial_{s}\partial_tx\,\partial_s y+\partial_{s}\partial_ty\, \partial_s x=-(\partial_{s}\partial_t\vec X)\cdot (\partial_s\vec X^{\perp}).
\end{equation}
Differentiating $W_c(t)$ defined in \eqref{AtWct} with respect to $t$, noting \eqref{dWdt continuous pre1}, we get
\begin{eqnarray}\label{dGamma(t) continuous gamma(theta)}
\frac{d}{dt}W_c(t)&=&\frac{d}{dt}\int_0^{L(t)}\gamma(\theta)ds
=\frac{d}{dt}\int_0^1
\gamma(\theta)\partial_\rho s\,d\rho\nonumber\\
	&=&\int_0^1\left(\gamma(\theta)\partial_{t}\partial_\rho s+\gamma'(\theta)\partial_t\theta\partial_\rho s\right)d\rho\nonumber\\
	&=&\int_0^1\left(\gamma(\theta)\partial_s\vec X\cdot \partial_s\partial_t\vec X-\gamma'(\theta)\partial_s\vec X^{\perp}\cdot \partial_s\partial_t\vec X\right)\partial_\rho s\,d\rho\nonumber\\
	&=&\Bigl( G(\theta)\partial_s \vec X,\partial_s\partial_t\vec X\Bigr)_{\Gamma(t)},\qquad t\ge0.
\end{eqnarray}
Taking the test functions  $\varphi=\mu$ in
\eqref{eqn:weak1 continuous torus} and $\boldsymbol{\omega}=\partial_t\vec X$ in \eqref{eqn:weak2 continuous torus}, we obtain
\begin{equation}\label{phi=mu continuous} \Bigl(\partial_{t}\vec{X},\mu\vec{n}\Bigr)_{\Gamma(t)}=-\Bigl(\partial_{s}\mu,
\partial_s\mu\Bigr)_{\Gamma(t)}, \qquad
\Bigl( G(\theta)\partial_s \vec X,\partial_s\partial_t\vec X\Bigr)_{\Gamma(t)}=\Bigl(\mu\vec{n},\partial_t\vec X\Bigr)_{\Gamma(t)},\qquad t\ge0.
\end{equation}
Substituting \eqref{phi=mu continuous} into
\eqref{dGamma(t) continuous gamma(theta)}, we have
\begin{eqnarray}\label{dWdt continuous}
\frac{d}{dt}W_c(t)&=&\Bigl( G(\theta)\partial_s \vec X,\partial_s\partial_t\vec X\Bigr)_{\Gamma(t)}
=\Bigl(\mu\vec{n},\partial_t\vec X\Bigr)_{\Gamma(t)}\nonumber\\
&=&-\Bigl(\partial_{s}\mu,
\partial_s\mu\Bigr)_{\Gamma(t)}\leq 0,\qquad t\ge0,
\end{eqnarray}
which immediately implies the energy dissipation in \eqref{eqn:massenergy2dd}.
\end{proof}

\section{A semi-discretization by PFEM and its properties}

In this section, we present a parametric finite element method (PFEM) with
conforming piecewise linear  elements to discretize the variational problem \eqref{eqn:weak continuous torus} and show that the semi-discretization
conserves area/mass and keeps energy dissipation.

\subsection{The semi-discretization in space}

Let $N>2$ be a positive integer and $h=1/N$ be the mesh size,
denote the grid points $\rho_{j}=jh$ for $j=0,1,\ldots,N$, and subintervals
$I_j=[\rho_{j-1},\rho_{j}]$ for $j=1,2,\ldots,N$. Then
a uniform partition of the interval $\mathbb{I}$ is given as
$\mathbb{I}=[0,1]=\bigcup_{j=1}^{N}I_j$. Introduce
the finite element subspaces
 \begin{eqnarray*}\label{eqn:H1 semi torus}
&&\mathbb{K}^h:=\{u^h\in C(\mathbb{I})\ |\ u^h\mid_{I_{j}}\in \mathcal{P}_1,\ \forall j=1,2,\ldots,N\}\subset \mathbb{K}, \\
&&\mathbb{K}_p^h:=\{u^h\in \mathbb{K}^h \ |\ u(0)=u(1)\}\subset \mathbb{K}_p,\qquad
\mathbb{X}_p^h:=\mathbb{K}_p^h\times \mathbb{K}_p^h\subset \mathbb{X}_p,
\end{eqnarray*}
where $\mathcal{P}_1$ denotes the space of all polynomials with degree at most $1$.

Let $\Gamma^h(t):=\vec{X}^h(\cdot,t)\in \mathbb{X}_p^h$ and
$\mu^h(t)\in \mathbb{K}_p^h$ be the numerical approximations of the closed curve $\Gamma(t):=\vec{X}(\cdot,t)\in \mathbb{X}_p$ and $\mu(\cdot,t)\in \mathbb{K}_p$, respectively, which is the solution of the variational problem
\eqref{eqn:weak continuous torus}. In fact, for $t\ge0$, the piecewise linear curve $\Gamma^h(t)$ is composed by ordered line segments $\{\vec h_j(t)\}_{j=1}^N$ and we  always assume that they satisfy
\begin{equation}\label{hjt987}
h_{\rm min}(t):=\min_{1\le j\le N} |\vec h_j(t)|>0, \quad \hbox{with}
\quad \vec h_j(t):=\vec X^h(\rho_j,t)-\vec X^h(\rho_{j-1},t), \quad j=1,2,\ldots, N,
\end{equation}
where $|\vec h_j(t)|$ denotes the length of the vector $\vec h_j(t)$ for $j=1,2,\ldots,N$. With the piecewise linear elements, it is easy to see that the unit tangential vector $\boldsymbol{\tau}^h$, the outward unit normal vector $\vec n^h$ and the inclination angle $\theta^h$ of the curve $\Gamma^h(t)$ are constant vectors/scalars on each interval $I_j$ with possible discontinuities or jumps at nodes $\rho_j$. In fact, for $1\leq j\leq N$, the two vectors $\boldsymbol{\tau}^h,\, \vec n^h$ on each interval $I_j$ can be computed as
\begin{equation}
\boldsymbol{\tau}^h|_{I_j}=\frac{\vec h_j}{|\vec h_j|}:=\boldsymbol{\tau}^h_j,\qquad
\vec{n}^h|_{I_j}
=-({\boldsymbol{\tau}}^h_j)^{\perp}=-\frac{(\vec h_j)^\perp}{|\vec h_j|}:=\vec{n}^h_j;
\label{eqn:semitannorm}
\end{equation}
and the angle $\theta^h$ on each interval $I_j$ is
\begin{equation}\label{semidiscretized theta}
\theta^h|_{I_j}:=\theta^h_j,\quad\hbox{satisfying}\quad  \cos\theta^h_j=\frac{h_{j,x}}{|\vec h_j|},\,\quad \sin\theta^h_j=\frac{h_{j,y}}{|\vec h_j|},\qquad \hbox{with}
\quad \vec h_j=
(h_{j,x},h_{j,y})^T.
\end{equation}
Furthermore, for two piecewise linear scalar (or vector) functions $u$ and $v$ defined on  $\mathbb{I}$ with possible jumps at the nodes $\{\rho_j\}_{j=0}^{N}$, we can define the mass lumped inner product $\big(\cdot,\cdot\big)_{\Gamma^h(t)}^h$ over $\Gamma^h(t)$ as
\begin{equation}
\big(u,~v\big)_{\Gamma^h(t)}^h:=\frac{1}{2}\sum_{j=1}^{N}|\vec h_j|\,\Big[\big(u\cdot v\big)(\rho_j^-)+\big(u\cdot v\big)(\rho_{j-1}^+)\Big],
\label{eqn:semiproduct}
\end{equation}
where $u(\rho_j^\pm)=\lim\limits_{\rho\to \rho_j^\pm} u(\rho)$ for $0\le j\le N$.

Let $\Gamma^h(0):=\vec{X}^h(\rho,0)\in \mathbb{X}_p^h$ be an interpolation
of the initial curve $\vec{X}_0(s)$ in \eqref{init} satisfying $\vec{X}_0(0)
=\vec{X}_0(L_0)$, which is defined as
$\vec{X}^h(\rho=\rho_j,0)=\vec{X}_0(s=s_j^0)$ with $s_j^0=L_0\rho_j$ for $j=0,1,\ldots,N$. Then a semi-discretization in space of the variational formulation \eqref{eqn:weak continuous torus} can be given as: Take $\Gamma^h(0)=\vec X^h(\cdot,0) \in \mathbb{X}_p^h$, find the closed curve $\Gamma^h(t):=\vec X^h(\cdot,t)=(x^h(\cdot,t),y^h(\cdot,t))^T\in \mathbb{X}_p^h$ and the weighted mean curvature $\mu^h(\cdot,~t)\in \mathbb{K}_p^h$, such that
\begin{subequations}
\label{eqn:2dsemi torus}
\begin{align}
\label{eqn:2dsemi1 torus}
&\Bigl(\partial_{t}\vec{X}^h,\varphi^h\vec{n}^h\Bigr)_{\Gamma^h(t)}^h+\Bigl(\partial_{s}\mu^h,
\partial_s\varphi^h\Bigr)_{\Gamma^h(t)}^h=0,\qquad\forall \varphi^h\in \mathbb{K}_p^h,\\[0.5em]\
\label{eqn:2dsemi2 torus}
&\Bigl(\mu^h,\vec{n}^h\cdot\boldsymbol{\omega}^h\Bigr)_{\Gamma^h(t)}^h-\Bigl( G(\theta^h)\partial_s \vec X^h,\partial_s\boldsymbol{\omega}^h\Bigr)_{\Gamma^h(t)}^h=0,
\quad\forall\boldsymbol{\omega}^h\in\mathbb{X}_p^h.
\end{align}
\end{subequations}

\subsection{Area/mass conservation and energy dissipation}

Let $A^h(t)$ be the area/mass of the region enclosed
by the closed curve $\Gamma^h(t)$ and $W^h_c(t)$ be its
total interfacial energy, which are defined as
\begin{equation}\label{discrete area}
A^h(t)=\frac{1}{2}\sum_{j=1}^N(x_j(t)-x_{j-1}(t))(y_j(t)+y_{j-1}(t)),
\qquad W^h_c(t)=\sum_{j=1}^N|\vec h_j(t)|\,\gamma(\theta_j^h),
\end{equation}
where $\vec X_j(t)=(x_j(t),y_j(t))^T:=\vec X^h(\rho_j,t)$ for $j=0,1,\ldots,N$.
For the semi-discretization  \eqref{eqn:2dsemi torus}, we have
\begin{prop}[area/mass conservation and energy dissipation]
Let $\Bigl(\vec X^h(\cdot,~t),~\mu^h(\cdot,~t)\Bigr)\in \mathbb{X}_p^h\times
\mathbb{K}_p^h$ be a solution of the semi-dicsretization
\eqref{eqn:2dsemi torus}. Then the area/mass $A^h(t)$ in
\eqref{discrete area} is conserved,  i.e.
\begin{equation}\label{massvp semi}
A^h(t)\equiv A^h(0)=\frac{1}{2}\sum_{j=1}^N[x_0(s_j^0)-
x_0(s_{j-1}^0)][y_0(s_j^0)+y_0(s_{j-1}^0)],\qquad t\geq 0;
\end{equation}
and the total interfacial energy $W^h_c(t)$ in
\eqref{discrete area} is dissipative, i.e.
\begin{equation}\label{massvp semieng}
W^h_c(t)\leq W^h_c(t_1)\leq W^h_c(0) = \sum_{j=1}^N|\vec h_j(0)|\,\gamma(\theta_j^h(0)),\qquad t\geq t_1\geq 0.
\end{equation}
\end{prop}

\begin{proof}
The area/mass conservation \eqref{massvp semi} of the semi-discretization \eqref{eqn:2dsemi torus} can be proved similar to those in  \citep[Proposition 2.1]{bao2017parametric} and thus the details
are omitted here for brevity.

Similar to the proof of \eqref{dWdt continuous pre1}, noting
\eqref{def of kappa by theta} and \eqref{semidiscretized theta}, we obtain
\begin{equation}\label{dWdt semi pre1}	\dot{\theta}_j^h(t)=\left(\sin^2(\theta_j^h)+\cos^2(\theta_j^h)\right)
\dot{\theta}_j^h(t)=\frac{-\dot{h}_{j,x}\,h_{j,y}+\dot{h}_{j,y}\,h_{j,x}}{|\vec h_j|^2}=-\frac{\vec h_j^{\perp}\cdot\dot{\vec h}_j }{|\vec h_j|^2}.
\end{equation}
Differentiating $W^h_c(t)$ in \eqref{discrete area} with respect to $t$,
noticing \eqref{usage of G} and \eqref{dWdt semi pre1}, we get
\begin{eqnarray}\label{dGamma(t) semi gamma(theta)}
\frac{d}{dt}W^h_c(t)&=&\frac{d}{dt}\left(\sum_{j=1}^N|\vec h_j(t)|\,\gamma(\theta_j^h)\right)=
\sum_{j=1}^N\left(\gamma(\theta_j^h)\frac{d}{dt}|{\vec h}_j(t)|+\gamma'(\theta_j^h)\,\dot{\theta}_j^h\,|\vec h_j(t)|\right)\nonumber\\
&=&\sum_{j=1}^N\left(\gamma(\theta^h_j)\frac{\vec h_j(t)\cdot \dot{\vec{h}}_j(t)}{|\vec h_j(t)|}-\gamma'(\theta^h_j)\frac{\vec h_j(t)^{\perp}\cdot \dot{\vec h}_j(t)}{|\vec h_j(t)|}\right)\nonumber\\
&=&\sum_{j=1}^N|\vec h_j(t)|\left(\gamma(\theta^h_j)\frac{\vec h_j(t)}{|\vec h_j(t)|}-\gamma'(\theta^h_j)\frac{\vec h_j(t)^{\perp}}{|\vec h_j(t)|}\right)\cdot \frac{\dot{\vec{h}}_j(t)}{|\vec h_j(t)|}\nonumber\\
&=&\sum_{j=1}^N|\vec h_j(t)|\, \left(G(\theta^h)\frac{\vec h_j(t)}{|\vec h_j(t)|}\right)\cdot \frac{\dot{\vec{h}}_j(t)}{|\vec h_j(t)|}\nonumber\\
&=&\sum_{j=1}^N|\vec h_j(t)|\, \left(G(\theta^h)\left.\partial_s \vec X^h\right|_{I_j}\right)\cdot \left(\left.\partial_s \partial_t\vec X^h\right) \right|_{I_j} \nonumber\\
&=&\Bigl( G(\theta^h)\partial_s \vec X^h,\partial_s\partial_t\vec X^h\Bigr)_{\Gamma^h(t)}^h.
\end{eqnarray}
Here we use the following equalities
\begin{equation}
\left.\partial_s \vec X^h\right|_{I_j}=\frac{\vec h_j(t)}{|\vec h_j(t)|},
\qquad \left.\partial_s \partial_t\vec X^h\right|_{I_j}=\frac{1}{|\vec h_j(t)|}\left. \partial_t\vec X^h\right|_{I_j}=\frac{\dot{\vec{h}}_j(t)}{|\vec h_j(t)|}, \qquad 1\le j\le N.
\end{equation}
Choosing the test functions $\varphi^h=\mu^h $ in \eqref{eqn:2dsemi1 torus} and $\boldsymbol{\omega}^h=\partial_t\vec X^h$ in \eqref{eqn:2dsemi2 torus}, we have
\begin{equation}\label{phi=mu semi-discretization}
\Bigl(\partial_t\vec X^h,~\mu^h\vec n^h\Bigr)_{\Gamma^h(t)}^h =- \Bigl(\partial_s\mu^h,~\partial_s\mu^h\Bigr)_{\Gamma^h(t)}^h, \qquad
\Bigl( G(\theta^h)\partial_s \vec X^h,\partial_s\partial_t\vec X^h\Bigr)_{\Gamma^h(t)}^h=\Bigl(\mu^h\vec{n}^h,\partial_t\vec X^h\Bigr)_{\Gamma^h(t)}^h.
\end{equation}
Substituting \eqref{phi=mu semi-discretization} into
\eqref{dGamma(t) semi gamma(theta)}, we get
\begin{eqnarray}\label{dWdt semi-discretization}
\frac{d}{dt}W^h_c(t)&=&\Bigl( G(\theta^h)\partial_s \vec X^h,\partial_s\partial_t\vec X^h\Bigr)_{\Gamma^h(t)}^h\nonumber\\
&=&\Bigl(\mu^h\vec{n}^h,\partial_t\vec X^h\Bigr)_{\Gamma^h(t)}^h
=\Bigl(\partial_t\vec X^h,~\mu^h\vec n^h\Bigr)_{\Gamma^h(t)}^h\nonumber\\
&=&-\Bigl(\partial_s\mu^h,~\partial_s\mu^h\Bigr)_{\Gamma^h(t)}^h\leq 0,\qquad t\ge0,
\end{eqnarray}
which immediately implies the energy dissipation in \eqref{massvp semieng}.
\end{proof}

\section{An energy-stable PFEM  and its properties}

In this section, we further discretize the semi-discretization \eqref{eqn:2dsemi torus} in time by a semi-implicit backward Euler method
to obtain a full-discretization of the variational problem
\eqref{eqn:weak continuous torus} (and thus of the original problem \eqref{eqn:original formulation gamma(theta)} or
\eqref{anisotropic surface diffusion, original}-\eqref{def of mu, original}),
establish its well-posedness and investigate some simple conditions
on $\gamma(\theta)$ such that the full-discretization is energy dissipative.

\subsection{The full-discretizition}

Take $\tau>0$ as the time step size and denote $t_m=m\tau$ the discrete time levels for $m=0,1,\ldots$ . For each $m\geq 0$, let $\Gamma^m:=\vec X^m(\rho)=(x^m(\rho),y^m(\rho))^T\in \mathbb{X}_p^h$ and $\mu^m\in \mathbb{K}_p^h$ be  the approximations of $\Gamma^h(t_m)=\vec X^h(\rho,t_m)$
and $\mu^h(t_m)\in \mathbb{K}_p^h$, respectively, which is
the solution of the semi-discretization \eqref{eqn:2dsemi torus}. Similarly, $\Gamma^m$ is composed by segments
$\{\vec h^m_j\}_{j=1}^N$ defined as
\begin{equation}
\vec h^m_j=(h_{j,x}^m,h_{j,y}^m)^T:=\vec X^m(\rho_j)-\vec X^m(\rho_{j-1}), \qquad  j=1,2,\ldots,N.
\end{equation}
Again, the unit tangential vector $\boldsymbol{\tau}^m$, the outward unit normal vector $\vec n^m$ and the inclination angle $\theta^m$ of the curve $\Gamma^m$ are constant vectors/scalars on each interval $I_j$ with possible discontinuities or jumps at nodes $\rho_j$ ($j=0,1,\ldots,N$). The two vectors $\boldsymbol{\tau}^m$ and $\vec n^m$ on interval $I_j$ can be computed as
\begin{equation}
\boldsymbol{\tau}^m|_{I_j}=\frac{\vec h_j^m}{|\vec h_j^m|}:=\boldsymbol{\tau}^m_j,\qquad
\vec{n}^m|_{I_j}
=-({\boldsymbol{\tau}}^m_j)^{\perp}=-\frac{(\vec h_j^m)^\perp}{|\vec h_j^m|}:=\vec{n}^m_j,\qquad 1\leq j\leq N,
\label{tangent and normal vector full case 1}
\end{equation}
and the angle $\theta^m$ on each interval $I_j$ is given as
\begin{equation}\label{discretized theta}
	\theta^m|_{I_j}:=\theta^m_j,\,\quad \hbox{satisfying}\quad \cos\theta^m_j=\frac{h_{j,x}^m}{|\vec h_j^m|},\,\quad\sin\theta^m_j=\frac{h_{j,y}^m}{|\vec h_j^m|}, \qquad 1\le j\le N.
\end{equation}

 Then an energy-stable PFEM ({\bf ES-PFEM}) to discretize the semi-discretization
\eqref{eqn:2dsemi torus} is to adapt a semi-implicit backward Euler method
in time and is give as:   Take $\Gamma^0:=\Gamma^h(0)\in \mathbb{X}_p^h$,
for $m\ge0$, find a closed curve $\Gamma^{m+1}:=\vec X^{m+1}(\cdot)=(x^{m+1}(\cdot),y^{m+1}(\cdot))^T\in \mathbb{X}_p^h$ and a weighted mean curvature $\mu^{m+1}(\cdot)\in \mathbb{K}_p^h$, such that
\begin{subequations}
\label{eqn:full discretization gamma(theta) torus}
\begin{align}
\label{eqn:full discretization gamma(theta) torus eq1}
&\Bigl(\frac{\vec X^{m+1}-\vec X^m}{\tau},~\varphi^h\vec n^m\Bigr)_{\Gamma^m}^h + \Bigl(\partial_s\mu^{m+1},~\partial_s\varphi^h\Bigr)_{\Gamma^m}^h=0,\qquad\forall \varphi^h\in \mathbb{K}^h_p,\\[0.5em]
\label{eqn:full discretization gamma(theta) torus eq2}
&\Bigl(\mu^{m+1},\vec n^m\cdot\boldsymbol{\omega}^h\Bigr)_{\Gamma^m}^h-\Bigl(G(\theta^m)\partial_s \vec X^{m+1},~\partial_s\boldsymbol{\omega}^h\Bigr)_{\Gamma^m}^h=0,\quad\forall\boldsymbol{\omega}^h\in\mathbb{X}^h_p.
\end{align}
\end{subequations}

 The above ES-PFEM is semi-implicit, i.e. only a linear system needs
to be solved at each time step, and thus it is very efficient.

\subsection{Well-posedness}

Assume $N\ge3$ and denote
\begin{equation}
\widetilde{\vec h}_{N-1}^m:=\vec X^m(\rho_{N-1})-\vec X^m(\rho_{1}),\quad  \widetilde{\vec h}_j^m:=\vec X^m(\rho_{j+1})-\vec X^m(\rho_{j-1}), \quad j=1,2,\ldots, N-2, \quad m\ge0.
\end{equation}
For the well-posedness of the full discretization ES-PFEM
\eqref{eqn:full discretization gamma(theta) torus}, we have

\begin{thm}[Well-posedness] For each $m\ge0$, assume that the following two conditions are satisfied

(i) at least two vectors in $\{\widetilde{\vec h}_j^m\}_{j=1}^{N-1}$ are not
parallel, i.e. there exists $1\le j_1<j_2\le N-1$ such that
\begin{equation}
\widetilde{\vec h}_{j_1}^m \cdot (\widetilde{\vec h}_{j_2}^m)^\perp\ne 0,
\end{equation}

(ii) no degenerate vertex on $\Gamma^m$, i.e.
\begin{equation}
h^m_{\rm min}:=\min_{1\le j\le N} |\vec h^m_j|=\min_{1\le j\le N} |\vec X^m(\rho_{j+1})-\vec X^m(\rho_{j})|\ >0.
\label{eqn:assumption}
\end{equation}
Then the full-discretization
\eqref{eqn:full discretization gamma(theta) torus} is well-posed, i.e.,
there exists a unique solution
$\bigl(\vec X^{m+1},~\kappa^{m+1}\bigr)\in \mathbb{X}^h_p\times \mathbb{K}^h_p$ of the problem
\eqref{eqn:full discretization gamma(theta) torus}.
\end{thm}

\begin{proof}
	We just need to prove the following homogeneous problem only has zero solution:
	\begin{subequations}
\label{eqn:full discretization gamma(theta) homo}
\begin{align}
\label{eqn:full discretization gamma(theta) homo eq1}
&\Bigl(\frac{\vec X^{m+1}}{\tau},~\varphi^h\vec n^m\Bigr)_{\Gamma^m}^h + \Bigl(\partial_s\mu^{m+1},~\partial_s\varphi^h\Bigr)_{\Gamma^m}^h=0,\qquad\forall \varphi^h\in \mathbb{K}^h_p,\\[0.5em]
\label{eqn:full discretization gamma(theta) homo eq2}
&\Bigl(\mu^{m+1},\vec n^m\cdot\boldsymbol{\omega}^h\Bigr)_{\Gamma^m}^h-\Bigl(G(\theta^m)\partial_s \vec X^{m+1},~\partial_s\boldsymbol{\omega}^h
\Bigr)_{\Gamma^m}^h=0,\qquad\forall\boldsymbol{\omega}^h\in\mathbb{X}^h_p.
\end{align}
\end{subequations}
Taking $\varphi^h=\mu^{m+1}$ in
\eqref{eqn:full discretization gamma(theta) homo eq1}, we get
\begin{equation}\label{phi=mu(m+1)}
	\Bigl(\frac{\vec X^{m+1}}{\tau},~\mu^{m+1}\vec n^m\Bigr)_{\Gamma^m}^h + \Bigl(\partial_s\mu^{m+1},~\partial_s\mu^{m+1}\Bigr)_{\Gamma^m}^h=0.
\end{equation}
Choosing $\boldsymbol{\omega}^h=\vec X^{m+1}$ in \eqref{eqn:full discretization gamma(theta) homo eq2}, we have
\begin{equation}\label{omega=X(m+1)}
\Bigl(\mu^{m+1},\vec n^m\cdot \vec X^{m+1}\Bigr)_{\Gamma^m}^h-\Bigl(G(\theta^m)\partial_s \vec X^{m+1},~\partial_s\vec X^{m+1}\Bigr)_{\Gamma^m}^h =0.
\end{equation}
Combining \eqref{phi=mu(m+1)} and \eqref{omega=X(m+1)}, noting
$G(\theta)$ is a positive definite matrix,  we obtain
\begin{eqnarray}\label{sum of squares, case gamma(theta)}
0&\le&\tau \Bigl(\partial_s\mu^{m+1},~\partial_s\mu^{m+1}
\Bigr)_{\Gamma^m}^h=-\Bigl(\vec X^{m+1},~\mu^{m+1}\vec n^m\Bigr)_{\Gamma^m}^h\nonumber \\
&=&-\Bigl(\mu^{m+1},\vec n^m\cdot \vec X^{m+1}\Bigr)_{\Gamma^m}^h\nonumber\\
&=&-\Bigl(G(\theta^m) \partial_s \vec X^{m+1},~\partial_s \vec X^{m+1}\Bigr)_{\Gamma^m}^h\le0.
\end{eqnarray}
Thus we have
\begin{equation}
\Bigl(\partial_s\mu^{m+1},~\partial_s\mu^{m+1}
\Bigr)_{\Gamma^m}^h=0, \qquad \Bigl(G(\theta^m) \partial_s \vec X^{m+1},~\partial_s \vec X^{m+1}\Bigr)_{\Gamma^m}^h=0,
\end{equation}
which yields
\begin{equation}
\label{Xs=0 gamma(theta)}
\partial_s\vec X^{m+1}\equiv \vec 0,\qquad
\partial_s\mu^{m+1}\equiv 0 \qquad \Rightarrow \qquad \vec X^{m+1}\equiv \vec X^c\in\mathbb{R}^2, \quad \mu^{m+1}\equiv \mu^c\in\mathbb{R}.
\end{equation}
Substituting \eqref{Xs=0 gamma(theta)} into
\eqref{eqn:full discretization gamma(theta) homo}, we obtain
\begin{subequations}
\label{Auxillary eq}
\begin{align}
\label{X=0 gamma(theta)}
&\Bigl(\frac{\vec X^c}{\tau},~\varphi^h\vec n^m\Bigr)_{\Gamma^m}^h=0,\qquad\forall \varphi^h\in \mathbb{K}^h_p,\\
\label{mu=0 gamma(theta)}
&\Bigl(\mu^{c},\vec n^m\cdot\boldsymbol{\omega}^h\Bigr)_{\Gamma^m}^h=0,
\qquad\forall\boldsymbol{\omega}^h\in\mathbb{X}^h_p.
\end{align}
\end{subequations}
Under the conditions (i) and (ii) and by using the Theorem 2.1 in \cite{barrett2007parametric}, we know that \eqref{Auxillary eq} implies $\mu^c=0$ and $\vec X^c=0$. Thus the homogeneous problem
\eqref{eqn:full discretization gamma(theta) homo} only has zero solution,
and thereby the original inhomogeneous linear system
\eqref{eqn:full discretization gamma(theta) torus} is well-posed, i.e. it has a unique solution.
\end{proof}

\subsection{Energy dissipation}

Define the total energy $W_c^m$ of the closed curve $\Gamma^m:=\vec X^{m}$ as
\begin{equation}\label{full-discretized energy gamma(theta) torus}
W_c^m:=W_c(\Gamma^m)=\sum_{j=1}^N|\vec h_j^m|\,\gamma(\theta_j^m), \quad m\ge0.
\end{equation}
We state a generic energy dissipation condition on $\gamma(\theta)\in C^1([-\pi,\pi])$ satisfying $\gamma(-\pi)=\gamma(\pi)$ as
%
\begin{equation}\label{desired condition gamma(theta)}
2\gamma(\theta)-\gamma(\theta)\cos(\theta-\phi)-
\gamma'(\theta)\sin(\theta-\phi)\geq \gamma(\phi),\qquad \forall \theta,\,\phi\in [-\pi,\pi],
\end{equation}
such that the ES-PFEM \eqref{eqn:full discretization gamma(theta) torus}
is unconditionally energy stable.

\begin{thm}[A generic condition for energy dissipation]
\label{Energy dissipation, theta thm}
Under the  condition \eqref{desired condition gamma(theta)} on
$\gamma(\theta)$, the ES-PFEM
\eqref{eqn:full discretization gamma(theta) torus} is unconditionally energy stable, i.e. for any $\tau>0$, we have
\begin{equation}\label{engdpfd}
W^{m+1}_c\leq W^m_c\leq  \ldots \le W^0_c=\sum_{j=1}^N|\vec h_j^0|\,\gamma(\theta_j^0),\qquad \forall m\ge0.
\end{equation}
\end{thm}

\begin{proof}
Taking $\varphi^h=\mu^{m+1}$ in
\eqref{eqn:full discretization gamma(theta) torus eq1}
and $\boldsymbol{\omega}^h=\vec X^{m+1}-\vec X^m$ in
\eqref{eqn:full discretization gamma(theta) torus eq2}, we get
\begin{subequations}
\label{energy dissipation theta step1}
\begin{align}
\label{energy disspation theta phi=mu}
&\Bigl(\frac{\vec X^{m+1}-\vec X^m}{\tau},~\mu^{m+1}\vec n^m\Bigr)_{\Gamma^m}^h + \Bigl(\partial_s\mu^{m+1},~\partial_s\mu^{m+1}\Bigr)_{\Gamma^m}^h=0;\\
\label{energy disspation theta omega=X-X}
&\Bigl(\mu^{m+1},\vec n^m\cdot(\vec X^{m+1}-\vec X^m)\Bigr)_{\Gamma^m}^h-\Bigl(G(\theta^m)\partial_s \vec X^{m+1},~\partial_s\vec X^{m+1}-\partial_s\vec X^m\Bigr)_{\Gamma^m}^h = 0.
\end{align}
\end{subequations}
Combining \eqref{energy dissipation theta step1},
\eqref{eqn:semiproduct} and \eqref{gthta1}, we have
\begin{align}\label{surface energy diff, gamma(theta) formulation temp}
&\Bigl(G(\theta^m)\partial_s \vec X^{m+1},~\partial_s\vec X^{m+1}-\partial_s\vec X^m\Bigr)_{\Gamma^m}^h+\int_{\Gamma^m}\gamma(\theta^m)ds\nonumber\\
&=\sum_{j=1}^N|\vec h_j^m|\left(\frac{\vec h^{m+1}_j}{|\vec h^m_j|}\cdot\frac{\vec h^{m+1}_j-\vec h^m_j}{|\vec h^m_j|}\right)\gamma(\theta_j^m)+\sum_{j=1}^N|\vec h_j^m|\left(\frac{(\vec h^{m+1}_j)^{\perp}}{|\vec h^{m+1}_j|}\cdot \frac{\vec h^{m}_j}{|\vec h^{m}_j|}\right)\gamma'(\theta_j^m)+\sum_{j=1}^N|\vec h_j^m|\,\gamma(\theta_j^m)\nonumber\\
&=\sum_{j=1}^N|\vec h_j^m|\left(\frac{|\vec h_j^{m+1}|\,\boldsymbol{\tau}^{m+1}_j}{|\vec h_j^m|}\cdot\frac{|\vec h_j^{m+1}|\,\boldsymbol{\tau}^{m+1}_j-|\vec h_j^m|\,\boldsymbol{\tau}^{m}_j}{|\vec h_j^m|}\right)\gamma(\theta^m_j)\nonumber\\
&\quad +\sum_{j=1}^N|\vec h_j^m|\left(\frac{|\vec h_j^{m+1}|\,(\boldsymbol{\tau}^{m+1}_j)^\perp}{|\vec h_j^m|}\cdot\frac{|\vec h_j^m|\,\boldsymbol{\tau}^{m}_j
}{|\vec h_j^m|}\right)\gamma'(\theta^m_j)+\sum_{j=1}^N|\vec h_j^m|\,\gamma(\theta^m_j)\\
&=\sum_{j=1}^N\frac{|\vec h_j^{m+1}|^2\gamma(\theta_j^m)+|\vec h_j^{m+1}|\,|\vec h_j^{m}|\left[\gamma'(\theta_j^m)\sin(\theta_j^{m+1}-\theta_j^m)
-\gamma(\theta_j^m)\cos(\theta_j^{m+1}-\theta_j^m)\right]
+|\vec h_j^{m}|^2\gamma(\theta_j^m)}{|\vec h^m_j|},\nonumber
\end{align}
where
\[
\boldsymbol{\tau}^{m+1}_j=\left(\cos\theta_j^{m+1},\sin\theta_j^{m+1}\right)^T,
\qquad \boldsymbol{\tau}^{m}_j=\left(\cos\theta_j^{m},\sin\theta_j^{m}\right)^T,
\qquad j=1,2,\ldots,N.
\]
By the inequality of arithmetic and geometric means, we have
\begin{equation}\label{AG ineq}
	|\vec h_j^{m+1}|^2\,\gamma(\theta_j^m)+|\vec h_j^{m}|^2\,\gamma(\theta_j^m)\geq 2|\vec h_j^{m}|\,|\vec h_j^{m+1}|\,\gamma(\theta_j^m), \qquad j=1,2,\ldots,N, \quad m\ge0.
\end{equation}
Combining \eqref{desired condition gamma(theta)}, \eqref{AG ineq} and \eqref{surface energy diff, gamma(theta) formulation temp}, we get
\begin{align}\label{surface energy diff, gamma(theta) formulation}
&\Bigl(G(\theta^m)\partial_s \vec X^{m+1},~\partial_s\vec X^{m+1}-\partial_s\vec X^m\Bigr)_{\Gamma^m}^h+\int_{\Gamma^m}\gamma(\theta^m)ds\nonumber\\
&\geq \sum_{j=1}^N\frac{|\vec h_j^{m+1}|\,|\vec h_j^{m}|\left[2\gamma(\theta_j^m)-\gamma(\theta_j^m)
\cos(\theta_j^{m+1}-\theta_j^m)+\gamma'(\theta_j^m)
\sin(\theta_j^{m+1}-\theta_j^m)\right]}{|\vec h^m_j|}\nonumber\\
&=\sum_{j=1}^N |\vec h_j^{m+1}|\left[2\gamma(\theta_j^m)-\gamma(\theta_j^m)
\cos(\theta_j^m-\theta_j^{m+1})-\gamma'(\theta_j^m)
\sin(\theta_j^m-\theta_j^{m+1})\right]\nonumber\\
&\geq \sum_{j=1}^N |\vec h_j^{m+1}|\gamma(\theta_j^{m+1})
=\int_{\Gamma^{m+1}}\gamma(\theta_j^{m+1})ds.
\end{align}
Combining the final result in
\eqref{surface energy diff, gamma(theta) formulation},
\eqref{energy disspation theta phi=mu} and
\eqref{energy disspation theta omega=X-X}, we have,
\begin{align}\label{energy disspation gamma(theta) final}
W^{m+1}_c-W^m_c&=\int_{\Gamma^{m+1}}\gamma(\theta^{m+1})ds
-\int_{\Gamma^m}\gamma(\theta^{m})ds\nonumber\\
&\leq\Bigl(G(\theta^m)\partial_s \vec X^{m+1},~\partial_s\vec X^{m+1}-\partial_s\vec X^m\Bigr)_{\Gamma^m}^h+\int_{\Gamma^m}\gamma(\theta^m)ds
-\int_{\Gamma^m}\gamma(\theta^{m})ds\nonumber\\
&=\Bigl(G(\theta^m)\partial_s \vec X^{m+1},~\partial_s\vec X^{m+1}-\partial_s\vec X^m\Bigr)_{\Gamma^m}^h\nonumber\\
&=\Bigl(\mu^{m+1},\vec n^m\cdot(\vec X^{m+1}-\vec X^m)\Bigr)_{\Gamma^m}^h\nonumber\\
&=-\tau\Bigl(\partial_s\mu^{m+1},
~\partial_s\mu^{m+1}\Bigr)_{\Gamma^m}^h\leq 0,\qquad m\ge0,
\end{align}
which immediately implies the energy dissipation in \eqref{engdpfd}.
\end{proof}

From the linearity and translation invariance with respect to $\gamma(\theta)$ in the energy dissipation condition \eqref{desired condition gamma(theta)}, we have

\begin{coro}[Addition, scalar multiplication and translation]
Assume that $\gamma_1(\theta),\,\gamma_2(\theta)$ be two anisotropic surface
energies satisfying  the energy dissipation condition
\eqref{desired condition gamma(theta)}, $\theta_0$ is a given constant. Then  $\gamma(\theta)=c\gamma_1(\theta)$ with $c>0$, $\gamma(\theta)=\gamma_1(\theta)+\gamma_2(\theta)$ and $\gamma(\theta)=\gamma_1(\theta-\theta_0)$  also
satisfy the energy dissipation condition
\eqref{desired condition gamma(theta)}.
\end{coro}

 Now we apply the result in Theorem \ref{Energy dissipation, theta thm}
to the ellipsoidal anisotropy surface energy \eqref{ellipsoidal} and obtain
a simple energy dissipation condition in this special case.

\begin{coro}[Ellipsoidal anisotropic surface energy]\label{coro: ellipsoidal}
For the ellipsoidal anisotropic surface energy $\gamma(\theta)$ in \eqref{ellipsoidal}, assume $-a/2\le b\le a$, then it satisfies
the energy dissipation condition
\eqref{desired condition gamma(theta)}, and thus the
ES-PFEM \eqref{eqn:full discretization gamma(theta) torus} is unconditionally energy stable.
\end{coro}

\begin{proof}Noticing that
$\gamma(\theta)=\sqrt{a+b\cos^2\theta}=\sqrt{a}\sqrt{1+\beta\cos^2\theta}$
with $\beta:=\frac{b}{a}$ and by Corollary 4.1, we need only to prove
the case when $a=1$ and $-1/2\le b=\beta\le 1$. Then we have
\begin{align*}
&2\gamma(\theta)-\gamma(\theta)\cos(\theta-\phi)-
\gamma'(\theta)\sin(\theta-\phi)-\gamma(\phi)\\
&=\frac{(2-\cos(\theta-\phi))\gamma^2(\theta)+
\beta\cos(\theta)\sin(\theta)\sin(\theta-\phi)-
\gamma(\theta)\gamma(\phi)}{\gamma(\theta)}\\
&\geq \frac{1}{\gamma(\theta)}\Bigl((2-\cos(\theta-\phi))
\gamma^2(\theta)+\beta\cos(\theta)\sin(\theta)\sin(\theta-\phi)
-\frac{\gamma^2(\theta)+\gamma^2(\phi)}{2}\Bigr)\\
&=\frac{1}{\gamma(\theta)}\sin ^2\left(\frac{\theta-\phi}{2}\right) (\beta \cos (\theta+\phi)+2 \beta \cos (2 \theta)+\beta+2)\\
&\geq \frac{1}{\gamma(\theta)}\sin ^2\left(\frac{\theta-\phi}{2}\right) \min\{2-2\beta, 2+4\beta\}\geq 0,\qquad \forall \theta,\,\phi\in [-\pi,\pi],
\end{align*}
which immediately implies $\gamma(\theta)$ satisfies the energy dissipation condition \eqref{desired condition gamma(theta)}.
\end{proof}

\begin{coro}[Riemannian  metric  anisotropic surface energy]\label{coro: BGN}
For the Riemannian  metric  anisotropic surface energy $\gamma(\theta)$ in \eqref{BGN}, assume $0<\lambda_{k}^{(1)}\le \lambda_{k}^{(2)}$ be the two eigenvalues of the symmetric positive definite matrix  $G_k$ for $k=1,2,\ldots, K$. If $\lambda_{k}^{(2)}\le 2\lambda_{k}^{(1)}$ for $k=1,2,\ldots, K$, then $\gamma(\theta)$ satisfies
the energy dissipation condition
\eqref{desired condition gamma(theta)}, and thus the
ES-PFEM \eqref{eqn:full discretization gamma(theta) torus} is unconditionally energy stable.
\end{coro}

\begin{proof}By Corollary 4.1, it suffices that we prove it is true when $K=1$. When $K=1$ in \eqref{BGN}, since $G_1$ is a symmetric positive definite matrix, thus there exists an orthonormal matrix (or a rotation matrix) $R_1\in {\mathbb R}^{2\times 2}$ such that
\begin{equation}\label{orthg}
R_1=\left(\begin{array}{cc} \cos\theta_1 & -\sin\theta_1\\ \sin\theta_1 &\cos\theta_1\\ \end{array}\right), \qquad
R_1^TG_1 R_1=\left(\begin{array}{cc} \lambda_{1}^{(1)} & 0\\ 0&\lambda_{1}^{(2)}\\ \end{array}\right),
\qquad \vec n(\theta)=R_1\vec n(\theta-\theta_1),
\end{equation}
where $\theta_1\in[-\pi,\pi)$ is a constant, $0<\lambda_{1}^{(1)}\le \lambda_{1}^{(2)}$ are the two eigenvalues of $G_1$, and $\vec n(\theta)$ is given in \eqref{BGN}.
Plugging \eqref{orthg} into \eqref{BGN} with $K=1$, we get
\begin{equation}
\gamma(\theta)=\sqrt{\vec n(\theta)^T G_1\vec n(\theta)}=
\sqrt{(\vec n(\theta-\theta_1))^T (R_1^T G_1 R_1) \vec n(\theta-\theta_1)}
=\sqrt{\lambda_{1}^{(1)}+(\lambda_{1}^{(2)}-\lambda_{1}^{(1)})
\cos^2(\theta-\theta_l)}.
\end{equation}
It is easy to see that $\lambda_{1}^{(2)}-\lambda_{1}^{(1)}\ge0\ge -\frac{1}{2}\lambda_{1}^{(1)}$.
Thus by Corollaries 4.2 and 4.1, when $\lambda_{1}^{(2)}-\lambda_{1}^{(1)}\le  \lambda_{1}^{(1)}$, i.e. $\lambda_{1}^{(2)}\le 2\lambda_{1}^{(1)}$, then $\gamma(\theta)$
satisfies the energy dissipation condition
\eqref{desired condition gamma(theta)}.
\end{proof}

Assume
\begin{equation}\label{fougma}
\gamma(\theta)=\frac{a_0}{2}+\sum\limits_{l=1}^\infty\left[a_l\cos( l\theta)+b_l\sin (l\theta)\right], \qquad -\pi\le \theta\le\pi,
\end{equation}
where $a_l$ ($l\ge0$) and $b_l$ ($l\ge1$) are the Fourier coefficients of
$\gamma(\theta)$, which are given as
\begin{equation}\label{gmfour1}
a_l=\frac{1}{\pi}\int_{-\pi}^\pi \gamma(\theta)
\cos( l\theta)d\theta, \qquad b_l=\frac{1}{\pi}\int_{-\pi}^\pi \gamma(\theta)
\sin( l\theta)d\theta, \qquad l\ge0.
\end{equation}
Then we can state a specific energy dissipation condition on $\gamma(\theta)$, which can be easily applied to the
$k$-fold anisotropy surface energy \eqref{kfold}.

\begin{thm}[A specific condition for energy dissipation]
\label{criteria for trignometric series lem}
Assume
\begin{equation}\label{criteria for trignometric series}
\frac{a_0}{2}\geq \sum_{l=1}^\infty(1+l^2)\sqrt{a_l^2+b_l^2},
\end{equation}
then the anisotropic surface energy $\gamma(\theta)$ satisfies the energy dissipation condition
\eqref{desired condition gamma(theta)}, and thus the ES-PFEM
\eqref{eqn:full discretization gamma(theta) torus} is unconditionally
energy stable.
\end{thm}

\begin{proof}
Under the assumption \eqref{criteria for trignometric series},
 we have
\begin{equation}\label{dgmtht34}
\gamma'(\theta)=\sum_{l=1}^{\infty}l\left[-a_l\sin (l\theta)+b_l\cos (l\theta)\right], \qquad \theta\in[-\pi,\pi].
\end{equation}
Plugging \eqref{gmfour1} and \eqref{dgmtht34}
into \eqref{desired condition gamma(theta)}, we
get
\begin{align*}
&2\gamma(\theta)-\gamma(\theta)\cos(\theta-\phi)
-\gamma(\phi)-\gamma'(\theta)\sin(\theta-\phi)\nonumber\\
&=\gamma(\theta)(1-\cos(\theta-\phi))+\gamma(\theta)
-\gamma(\phi)-\gamma'(\theta)\sin(\theta-\phi)\nonumber\\
&=\gamma(\theta)(1-\cos(\theta-\phi))+\sum_{l=1}^{\infty}\Bigl[a_l[\cos (l\theta)-\cos (l\phi)+l\sin (l\theta)\sin(\theta-\phi)]\bigr.\nonumber\\
&\qquad\qquad\qquad\qquad\qquad\,+\Bigl.b_l[\sin (l\theta)-\sin(l\phi)-l\cos (l\theta)\sin(\theta-\phi)]\Bigr].\nonumber
\end{align*}
By Lemma \ref{trignometric ineq coro} in Appendix B, noting
\eqref{criteria for trignometric series}, we have
\begin{align}\label{proof for trigonometric series part 2}
&2\gamma(\theta)-\gamma(\theta)\cos(\theta-\phi)-\gamma(\phi)
-\gamma'(\theta)\sin(\theta-\phi)\nonumber\\
&\geq \gamma(\theta)(1-\cos(\theta-\phi))-
\sum_{l=1}^{\infty}\left((1-\cos(\theta-\phi))l^2
\sqrt{a_l^2+b_l^2}\right)\nonumber\\
&=(1-\cos(\theta-\phi))\left(\frac{a_0}{2}+\sum_{l=1}^\infty \left(a_l\cos (l\theta)+b_l\sin (l\theta)-l^2\sqrt{a_l^2+b_l^2}\right)\right)\nonumber\\
&\geq (1-\cos(\theta-\phi))\left(\frac{a_0}{2}-
\sum_{l=1}^\infty(1+l^2)\sqrt{a_l^2+b_l^2}\right)\nonumber\\
&\geq 0,\qquad \forall \theta,\,\phi\in [-\pi,\pi],
\end{align}
which immediately implies $\gamma(\theta)$ satisfies the energy dissipation condition \eqref{desired condition gamma(theta)}.
\end{proof}

By using Theorem \ref{criteria for trignometric series lem}, we can find a sufficient condition  on the $k$-fold anisotropy energy $\gamma(\theta)$
in \eqref{kfold} so that it satisfies the energy dissipation condition
\eqref{desired condition gamma(theta)}. Furthermore, we can prove that the condition is also necessary in this special case, which implies that our Theorem 4.3 is sharp and can hardly be improved.

\begin{coro}[$k$-fold anisotropic surface energy]\label{coro: k-fold}
For the $k$-fold anisotropy energy $\gamma(\theta)$
in \eqref{kfold}, it satisfies the energy dissipation condition \eqref{desired condition gamma(theta)} if and only if
\begin{equation} \label{beta11}
|\beta|\leq \beta_{\max}:=\frac{1}{1+k^2}.
\end{equation}
\end{coro}

\begin{proof} ($\Leftarrow$)
Combining \eqref{kfold} and \eqref{fougma}, we have
\begin{equation}\label{a0albl}
a_0=2,\quad a_k=\beta \cos(k\theta_0), \quad
b_k=\beta \sin (k\theta_0), \qquad a_l=b_l=0, \quad 1\le l\ne k.
\end{equation}
Combining \eqref{a0albl} and \eqref{criteria for trignometric series},
under the condition \eqref{beta11},
we obtain
\begin{equation}
\sum_{l=1}^\infty(1+l^2)\sqrt{a_l^2+b_l^2}=(1+k^2)\sqrt{a_k^2+b_k^2}=
(1+k^2)|\beta|\le (1+k^2)\frac{1}{1+k^2}=1=\frac{a_0}{2},
\end{equation}
which implies that \eqref{criteria for trignometric series} is satisfied
and thus the energy dissipation condition
\eqref{desired condition gamma(theta)} is satisfied.

\medskip

($\Rightarrow$) Denote \begin{equation}\label{fphi}
g(\phi)=2\gamma(\theta)-\gamma(\theta)\cos(\theta-\phi)
-\gamma'(\theta)\sin(\theta-\phi)-\gamma(\phi), \qquad \phi\in[-\pi,\pi].
\end{equation}
Differentiating \eqref{fphi} with respect to $\phi$, we have
\begin{subequations}
\label{fddphi}
\begin{align}
\label{fddphi1}
&g^\prime(\phi)=-\gamma(\theta)\sin(\theta-\phi)
+\gamma'(\theta)\cos(\theta-\phi)-\gamma^\prime(\phi),\\
&g^{\prime\prime}(\phi)=\gamma(\theta)\cos(\theta-\phi)
+\gamma'(\theta)\sin(\theta-\phi)-\gamma^{\prime\prime}(\phi).
\end{align}
\end{subequations}
Taking $\phi=\theta$ in \eqref{fphi} and \eqref{fddphi},
we obtain
\begin{equation} \label{fphi89}
g(\theta)=g'(\theta)=0, \qquad g''(\theta)=\gamma(\theta)-\gamma''(\theta), \qquad \theta\in[-\pi,\pi].
\end{equation}
The energy dissipation \eqref{desired condition gamma(theta)} implies
\begin{equation}\label{fphi67}
g(\phi)\geq 0,\qquad \forall \theta,\,\phi\in [-\pi,\pi].
\end{equation}
Combining \eqref{fphi67} and \eqref{fphi89}, noticing
\eqref{kfold}, we have
\begin{equation}
0\leq g''(\phi)|_{\phi=\theta}=g''(\theta)=\gamma(\theta)-\gamma''(\theta)=
1+(1+k^2)\beta \cos(k(\theta-\theta_0)),
\qquad \forall \theta\in [-\pi,\pi],
\end{equation}
which immediately implies the condition \eqref{beta11}.
\end{proof}

If $\gamma(\theta)\in C^3([-\pi,\pi])$, we can state another specific energy dissipation condition on $\gamma(\theta)$.

\begin{thm}[Another specific condition for energy dissipation]
If $\gamma(\theta)\in C^3([-\pi,\pi])$ satisfies
\begin{equation}\label{sharp criteria}
	\int_{-\pi}^{\pi}\frac{\gamma(\theta)}{2\pi}d\theta\geq \frac{5}{2}\left\|\gamma^{(3)}(\theta)\right\|_{L^2},
\end{equation}
then it satisfies the energy dissipation condition
\eqref{desired condition gamma(theta)}, and thus the ES-PFEM
\eqref{eqn:full discretization gamma(theta) torus} is unconditionally
energy stable.
\end{thm}

\begin{proof} Plugging \eqref{fougma} into the left-hand of
\eqref{sharp criteria}, we get
\begin{equation}\label{auxiliary eq0}
\int_{-\pi}^{\pi}\frac{\gamma(\theta)}{2\pi}d\theta=\frac{a_0}{2}.
\end{equation}
Similarly, plugging \eqref{fougma} into the right-hand of
\eqref{sharp criteria} and applying the Cauchy-Schwarz inequality, we get
\begin{align}\label{auxiliary ineq1}
\left\|\gamma^{(3)}(\theta)\right\|_{L^2}&=
\left(\sum_{l=1}^{\infty}l^6(a_l^2+b_l^2)\right)^{1/2}\nonumber\\
&=\left(\sum_{l=1}^{\infty}\frac{(1+l^2)^2}
{l^6}\right)^{-1/2}\left[\left(\sum_{l=1}^{\infty}l^6(a_l^2+b_l^2)
\right)^{1/2}\left(\sum_{l=1}^{\infty}\frac{(1+l^2)^2}{l^6}
\right)^{1/2}\right]\nonumber\\
&\geq \left(\sum_{l=1}^{\infty}\frac{(1+l^2)^2}{l^6}
\right)^{-1/2}\left(\sum_{l=1}^{\infty}(1+l^2)\sqrt{a_l^2+b_l^2}\right).
\end{align}
By using Fourier series, we have the following estimate
\begin{equation}\label{auxiliary ineq2} \sum_{l=1}^{\infty}\frac{(1+l^2)^2}{l^6}=
\frac{2\pi^4+42\pi^3+315\pi^2}{1890}\leq \frac{25}{4}\quad \Rightarrow\quad \left(\sum_{l=1}^{\infty}\frac{(1+l^2)^2}
{l^6}\right)^{-1/2}\ge \frac{2}{5}.
\end{equation}
Combining \eqref{auxiliary eq0}, \eqref{auxiliary ineq1} and
\eqref{auxiliary ineq2}, we obtain
\begin{equation}	\frac{a_0}{2}=\int_{-\pi}^{\pi}\frac{\gamma(\theta)}{2\pi}d\theta\geq \frac{5}{2}\left\|\gamma^{(3)}(\theta)\right\|_{L^2}\geq \sum_{l=1}^{\infty}(1+l^2)\sqrt{a_l^2+b_l^2},
\end{equation}
which immediately implies \eqref{criteria for trignometric series} is satisfied, and thus \eqref{desired condition gamma(theta)} is satisfied
by using Theorem \ref{criteria for trignometric series lem}.
\end{proof}

\section{Extension to solid-state dewetting}

In this section, we extend the new and simple variational formulation
\eqref{eqn:weak continuous torus} and its ES-PFEM
\eqref{eqn:full discretization gamma(theta) torus}
for a closed curve under anisotropic surface diffusion to
solid-state dewetting in materials science \cite{wang2015sharp,jiang2016solid,bao2017parametric},
i.e. evolution of an open
curve under anisotropic surface diffusion and contact
line migration (cf. Figure \ref{fig:model2d open}).

\begin{figure}[http]
\centering
\includegraphics[width=0.6\textwidth]{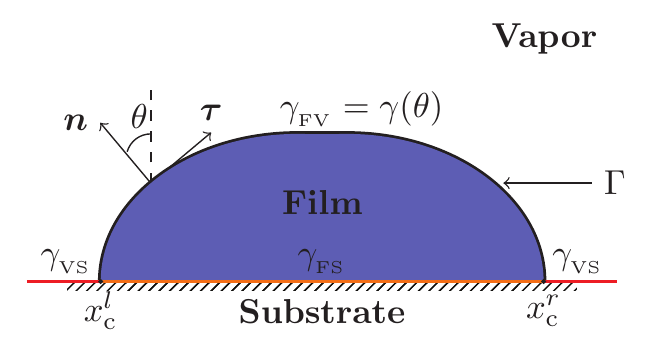}
\caption{A schematic illustration of a thin film on a rigid, flat substrate (i.e., the $x$-axis) in two dimensions,
where $x_c^l$ and $x_c^r$ are the left
and right contact points, $\gamma_{_{\subFV}}=\gamma(\theta),~\gamma_{_{\subVS}}$ and $\gamma_{_{\subFS}}$ represent the film vapor, vapor substrate and film substrate surface energy densities, respectively.}
\label{fig:model2d open}
\end{figure}

\subsection{A sharp interface model and its new variational formulation}

As shown in Figure \ref{fig:model2d open}, a typical problem in solid-state dewetting is to study the motion of an open curve $\Gamma:=\Gamma(t)$
under anisotropic surface diffusion with its two contact points
$x_c^l:=x_c^l(t)$ and $x_c^r:=x_c^r(t)$ moving along the rigid flat
substrate. By adapting the same notations in the previous sections
except removing the periodic boundary conditions, we represent
$\Gamma(t):=\vec X(\rho,t)=(x(\rho,t),y(\rho,t))^T$
for $0\le \rho\le 1$ (or respectively, $\Gamma(t):=\vec X(s,t)=(x(s,t),y(s,t))^T$ with $0\le s\le L(t)$ the arc length parameter and
$L(t)$ the length of $\Gamma(t)$). As it was derived in the literature \cite{wang2015sharp,jiang2016solid,bao2017parametric},
$\vec X(\rho,t)$ satisfies the anisotropic surface diffusion \eqref{anisotropic surface diffusion, original}-\eqref{def of mu, original}
and the following boundary conditions: \cite{wang2015sharp,jiang2016solid,bao2017parametric}
\begin{itemize}
\item [(i)] contact point condition
\begin{equation}
\label{eqn:weakBC1 gamma(theta) open}
y(0,t)=0,\quad y(1,t) = 0,\quad t\geq0;
\end{equation}
\item [(ii)] relaxed contact angle condition
\begin{equation}
\label{eqn:weakBC2 gamma(theta) open}
\frac{d x_c^l(t)}{d t}=\eta\, f(\theta_d^l;\sigma),\qquad
\frac{d x_c^r(t)}{d t}=-\eta\, f(\theta_d^r;\sigma),\qquad t\ge0;
\end{equation}
\item [(iii)]zero-mass flux condition
\begin{equation}
\label{eqn:weakBC3 gamma(theta) open}
\partial_s\mu(0,t) =0,\qquad\partial_s\mu(1,t) = 0, \qquad t\geq 0;
\end{equation}
\end{itemize}
satisfying $x_c^l(t)=x(0,t)\le x_c^r(t)=x(1,t)$, where $\theta_d^l:=\theta_d^l(t)$ and $\theta_d^r:=\theta_d^r(t)$ are the contact angles at the left and the right contact points, respectively. $0<\eta<\infty$ denotes the contact line mobility and $f(\theta;\sigma)$ is defined as
\begin{equation}\label{def of f} f(\theta;\sigma)=\gamma(\theta)\cos\theta-\gamma^\prime(\theta)\sin\theta-\sigma, \quad\theta\in [-\pi,\pi],
\end{equation}
with $\sigma=\cos\theta_i=\frac{\gamma_{_{\subVS}}-\gamma_{_{\subFS}}}{\gamma_0}$
and $\theta_i$ and $\gamma_0$ being the isotropic Young contact angle and dimensionless surface energy unit, respectively \cite{wang2015sharp,jiang2016solid,bao2017parametric}.
The initial condition is given as \eqref{init} satisfying $y_0(0)=y_0(L_0)=0$ and  $x_c^l(0)=x_0(0)\le x_c^r(0)=x_0(L_0)$ with $L_0$ the length of the curve at $t=0$.

Let $A(t)$ (defined in \eqref{AtWct}) be the area/mass of the region enclosed by $\Gamma(t)$ and the flat substrate, and define the total interfacial
energy $W_o(t)$ as
\begin{equation}
\label{Energy open}
W_o(t)=\int_{\Gamma(t)}\gamma(\theta)ds-\sigma(x_c^r(t)-x_c^l(t)), \qquad t\ge0.
\end{equation}
As it was proven in the literature \cite{wang2015sharp,jiang2016solid,bao2017parametric}, we have
\citep{bao2017parametric}
\begin{equation}
\label{eqn:two properties pde open}
\frac{d}{dt}A(t)=0,\qquad
\frac{d}{dt}W_o(t)=-\int_{\Gamma(t)}|\partial_s \mu|^2ds-\frac{1}{\eta}\Bigl[\left(\frac{dx_c^r(t)}{dt}\right)^2
+\left(\frac{dx_c^l(t)}{dt}\right)^2\Bigr]\leq 0, \qquad t\ge0,
\end{equation}
which implies area/mass conservation and energy dissipation, i.e.
\begin{equation}
A(t) \equiv A(0),\qquad W_o(t)\le W_o(t_1)\le W_o(0)=\int_{\Gamma(0)}\gamma(\theta)\,ds-\sigma(x_c^r(0)-x_c^l(0)), \qquad t\ge t_1\ge0.
\label{eqn:massenergy2ddopen}
\end{equation}

Introduce the functional spaces
\begin{equation}
\label{eqn:H01 continuous}
H^1_0(\mathbb{I}):=\{u\in H^1(\mathbb{I})\ |\  u(0)=u(1)=0\},
\qquad \mathbb{X}:=H^1(\mathbb{I})\times H_0^1(\mathbb{I}).
\end{equation}
Similar to those derivations in Section 2, we can obtain a
new and simple variational formulation for
\eqref{eqn:original formulation gamma(theta)} with
the boundary conditions
\eqref{eqn:weakBC1 gamma(theta) open}-\eqref{eqn:weakBC3 gamma(theta) open}
and the initial condition \eqref{init} as:
Given an initial open curve $\Gamma(0):=\vec X(\cdot,0)=\vec X_0\in \mathbb{X}$,
 find an open curve $\Gamma(t)=\vec X(\cdot, t)\in \mathbb{X}$ and
$\mu(t)\in \mathbb{K}$, such that:
\begin{subequations}
\label{eqn:weak continuous open}
\begin{align}
\label{eqn:weak1 continuous open}
&\Bigl(\partial_{t}\vec{X},\varphi\vec{n}\Bigr)_{\Gamma(t)}+\Bigl(\partial_{s}\mu,
\partial_s\varphi\Bigr)_{\Gamma(t)}=0,\qquad\forall \varphi\in \mathbb{K},\\[0.5em]
\label{eqn:weak2 continuous open}
&\Bigl(\mu,~\vec{n}\cdot\boldsymbol{\omega}\Bigr)_{\Gamma(t)}-\Bigl( G(\theta)\partial_s \vec X,\partial_s\boldsymbol{\omega}\Bigr)_{\Gamma(t)}-
\frac{1}{\eta}\Bigr[\frac{dx_c^l(t)}{dt}\,\omega_1(0) + \frac{dx_c^r(t)}{dt}\omega_1(1)\Bigr] \nonumber \\
&\qquad\qquad\qquad\qquad + \sigma\left[\omega_1(1)-\omega_1(0)\right] = 0,\quad\forall\boldsymbol{\omega}=(\omega_1,\omega_2)^T\in\mathbb{X},
\end{align}
\end{subequations}
satisfying  $x_c^l(t)=x(0,t)\le x_c^r(t)=x(1,t)$.

Similar to the closed curve case, one can easily show
area/mass conservation and energy dissipation of
the variational problem \eqref{eqn:weak continuous open},
i.e. \eqref{eqn:massenergy2ddopen} is valid.
The details are omitted here for brevity.

\subsection{An energy-stable PFEM and its properties}

Introduce the finite element subspaces
 \begin{equation}
\label{eqn:H1 semi open}
\mathbb{K}_0^h:=\{u^h\in\mathbb{K}^h \ |\ u^h(0)=u^h(1)=0\},
\qquad \mathbb{X}^h:=\mathbb{K}^h\times \mathbb{K}_0^h.
\end{equation}
Similar to Section 3, we can discretize \eqref{eqn:weak continuous open}
in space by PFEM and establish its area/mass conservation and energy dissipation of the semi-discrtization. Again, the details are omitted here
for brevity.

For each $m\geq 0$, let $\Gamma^m:=\vec X^m(\rho)=(x^m(\rho),y^m(\rho))^T\in \mathbb{X}^h$ and $\mu^m\in \mathbb{K}^h$ be  the approximations of $\Gamma(\cdot,t_m)=\vec X(\rho,t_m)\in \mathbb{X}$
and $\mu(\cdot,t_m)\in \mathbb{K}$, respectively, which is
the solution of the variational problem \eqref{eqn:weak continuous open}
with the initial data \eqref{init}. Let $\Gamma^0:=\vec{X}^0(\rho)=(x^0(\rho),y^0(\rho))^T\in \mathbb{X}^h$ be an interpolation of the initial curve $\vec{X}_0(s)$ in \eqref{init}, which is defined as $\vec{X}^0(\rho=\rho_j)=\vec{X}_0(s=s_j^0)$
with $s_j^0=L_0\rho_j$ for $j=0,1,\ldots,N$.
Then an energy-stable PFEM ({\bf ES-PFEM}) for discretizing
\eqref{eqn:weak continuous open} with \eqref{init} is given as:
Take $\Gamma^0=\vec X^0(\cdot) \in \mathbb{X}^h$ satisfying $x^0(0)\le x^0(1)$ and $y^0(0)=y^0(1)=0$, and set $x_l^0:=x^0(0)$ and $x_r^0:=x^0(1)$, for $m\ge0$,  find $\Gamma^{m+1}=\vec X^{m+1}(\cdot)=\left(x^{m+1}(\cdot),y^{m+1}(\cdot)\right)^T\in
\mathbb{X}^h$
and  $\mu^{m+1}(\cdot)\in \mathbb{K}^h$, such that
\begin{subequations}
\label{eqn:full discretization gamma(theta) open}
\begin{align}
\label{eqn:full discretization gamma(theta) open eq1}
&\Bigl(\frac{\vec X^{m+1}-\vec X^m}{\tau},~\varphi^h\vec n^m\Bigr)_{\Gamma^m}^h + \Bigl(\partial_s\mu^{m+1},~\partial_s\varphi^h\Bigr)_{\Gamma^m}^h=0,\qquad\forall \varphi^h\in \mathbb{K}^h,\\[0.5em]
\label{eqn:full discretization gamma(theta) open eq2}
&\Bigl(\mu^{m+1},\vec n^m\cdot\boldsymbol{\omega}^h\Bigr)_{\Gamma^m}^h-\Bigl(G(\theta^m)\partial_s \vec X^{m+1},~\partial_s\boldsymbol{\omega}^h\Bigr)_{\Gamma^m}^h - \frac{1}{\eta}\Bigr[\frac{x_l^{m+1}-x_l^n}{\tau}\,\omega_1^h(0) + \frac{x_r^{m+1}-x_r^m}{\tau}\omega_1^h(1)\Bigr]\nonumber\\
&\qquad\qquad\qquad\qquad + \;\sigma\,\Bigl[\omega_1^h(1) - \omega_1^h(0)\Bigr] = 0,\quad\forall\boldsymbol{\omega}^h\in\mathbb{X}^h,
\end{align}
\end{subequations}
satisfying $x_l^{m+1}=x^{m+1}(0)\le x_r^{m+1}=x^{m+1}(1)$.

The above ES-PFEM is semi-implicit, i.e. only a linear system needs
to be solved at each time step, and thus it is very efficient.
We have the following result for its well-posedness.

\begin{thm}[Well-posedness] For each $m\ge0$, assume the condition
\eqref{eqn:assumption} is valid and
at leat one of  ${\vec h}_1^m$ and ${\vec h}_N^m$ is not horizontal, i.e.
\begin{equation}\label{stab976}
({\vec h}_1^m\cdot {\vec e}_2)^2+({\vec h}_N^m\cdot {\vec e}_2)^2>0,
\qquad \hbox{with}\quad  {\vec e}_2=(0,1)^T.
\end{equation}
Then the full-discretization
\eqref{eqn:full discretization gamma(theta) open} is well-posed, i.e.,
there exists a unique solution $\bigl(\vec X^{m+1},\kappa^{m+1}\bigr)\in\bigl(\mathbb{X}^h,\mathbb{K}^h\bigr)$.
\end{thm}

\begin{proof}
Again, we just need to prove the following homogeneous problem only has zero solution:
\begin{subequations}
\label{eqn:full discretization gamma(theta) homo open}
\begin{align}
\label{eqn:full discretization gamma(theta) homo eq1 open}
&\Bigl(\frac{\vec X^{m+1}}{\tau},~\varphi^h\vec n^m\Bigr)_{\Gamma^m}^h + \Bigl(\partial_s\mu^{m+1},~\partial_s\varphi^h\Bigr)_{\Gamma^m}^h=0,\qquad\forall \varphi^h\in \mathbb{K}^h,\\[0.5em]
\label{eqn:full discretization gamma(theta) homo eq2 open}
&\Bigl(\mu^{m+1},\vec n^m\cdot\boldsymbol{\omega}^h\Bigr)_{\Gamma^m}^h-\Bigl(G(\theta^m)\partial_s \vec X^{m+1},~\partial_s\boldsymbol{\omega}^h\Bigr)_{\Gamma^m}^h
- \frac{x_l^{m+1}\,\omega_1^h(0) + x_r^{m+1}\,\omega_1^h(1)}{\eta\tau}=0,
\quad  \forall\boldsymbol{\omega}^h\in\mathbb{X}^h.
\end{align}
\end{subequations}
Taking $\varphi^h=\mu^{m+1}$ in
\eqref{eqn:full discretization gamma(theta) homo eq1 open} and $\boldsymbol{\omega}^h=\vec X^{m+1}$ in
\eqref{eqn:full discretization gamma(theta) homo eq2 open},
multiplying the first one by $\tau$, and then subtracting it by the second one, we obtain
\begin{equation}\label{sum of squares, case gamma(theta) open}
\tau \Bigl(\partial_s\mu^{m+1},~\partial_s\mu^{m+1}\Bigr)_{\Gamma^m}^h
+\Bigl(\gamma(\theta^m) \partial_s \vec X^{m+1},~\partial_s \vec X^{m+1}\Bigr)_{\Gamma^m}^h +\frac{(x_l^{m+1})^2+ (x_r^{m+1})^2}{\eta\tau}=0.
\end{equation}
Since $G(\theta)$ is a positive definite matrix and thus
the left hand side of
\eqref{sum of squares, case gamma(theta) open}
is the summation of squares, we obtain
\begin{equation}
\label{zeros case gamma(theta) open}
\partial_s\vec X^{m+1}\equiv \vec 0,\qquad
\partial_s\mu^{m+1}\equiv  0,\qquad
x_l^{m+1}=0,\qquad x_r^{m+1}=0.
\end{equation}
This, together with $\vec X^{m+1}\in \mathbb{X}^h$ and $\mu^{m+1}\in\mathbb{K}^h$,
implies that
\begin{equation} \label{xmp1976}
\vec X^{m+1}\equiv \vec 0, \qquad
\mu^{m+1}\equiv \mu^c\in\mathbb{R}.
\end{equation}
Substituting \eqref{xmp1976} into
\eqref{eqn:full discretization gamma(theta) homo eq2 open}, we obtain
\begin{equation}
\label{mu=0 gamma(theta) open}
\Bigl(\mu^{c},\vec n^m\cdot\boldsymbol{\omega}^h\Bigr)_{\Gamma^m}^h=0,
\qquad\forall\boldsymbol{\omega}^h=(\omega_1^h,~\omega_2^h)^T\in\mathbb{X}^h.
\end{equation}
Under the assumptions \eqref{eqn:assumption} and \eqref{stab976} and
by using Theorem 4.1 in \cite{bao2020energy}, then \eqref{mu=0 gamma(theta) open} implies $\mu^c=0$. Thus the homogeneous problem
\eqref{eqn:full discretization gamma(theta) homo open} only has zero solution, and thereby the original inhomogeneous linear system \eqref{eqn:full discretization gamma(theta) open} is
well-posed.
\end{proof}

Define the total interfacial energy of the
open polygonal curve $\Gamma^m$ as
\begin{equation}\label{full-discretized energy gamma(theta) open}
	W^m_o:=W_o(\Gamma^m)=\sum_{j=1}^N|\vec h_j^m|\,\gamma(\theta_j^m)-\sigma(x_r^m-x_l^m), \qquad m\ge0.
\end{equation}
Then for the ES-PFEM \eqref{eqn:full discretization gamma(theta) open},
we have the following results on its energy dissipation.

\begin{thm}[Energy dissipation]\label{Energy dissipation, theta thm open}
Under the  condition \eqref{desired condition gamma(theta)} on
$\gamma(\theta)$, the ES-PFEM
\eqref{eqn:full discretization gamma(theta) open}
is unconditionally energy-stable, i.e. for any $\tau>0$, we have
\begin{equation}\label{edppopen45}
W^{m+1}_o\leq W^m_o\leq \ldots\le W^0_o:=\sum_{j=1}^N|\vec h_j^0|\,\gamma(\theta_j^0)-\sigma(x_r^0-x_l^0),\qquad \forall m\ge0.
\end{equation}
\end{thm}

\begin{proof}
Taking $\varphi^h=\mu^{m+1}$ in
\eqref{eqn:full discretization gamma(theta) open eq1} and $\boldsymbol{\omega}^h=\vec X^{m+1}-\vec X^m$ in
\eqref{eqn:full discretization gamma(theta) open eq2}, multiplying the first one by $\tau$, and then subtracting it by the second one, we obtain
\begin{align}
\label{energy disspation theta open}
&\tau \Bigl(\partial_s\mu^{m+1},~\partial_s\mu^{m+1}
\Bigr)_{\Gamma^m}^h+\Bigl(G(\theta^m)\partial_s \vec X^{m+1},~\partial_s\vec X^{m+1}-\partial_s\vec X^m\Bigr)_{\Gamma^m}^h\nonumber\\
&\qquad\qquad +\frac{1}{\eta}\Bigr[\frac{(x_l^{m+1}-x_l^m)^2}{\tau}\, + \frac{(x_r^{m+1}-x_r^m)^2}{\tau}\Bigr] - \;\sigma\,\Bigl[(x_r^{m+1}-x_r^m) - (x_l^{m+1}-x_l^m)\Bigr] = 0.
\end{align}
Under the condition \eqref{desired condition gamma(theta)} and noting \eqref{surface energy diff, gamma(theta) formulation}, we get
\begin{align}\label{energy disspation gamma(theta) final open}
W^{m+1}_o-W^m_o&=\int_{\Gamma^{m+1}}\gamma(\vec n^{m+1})ds-\sigma(x_r^{m+1}-x_l^{m+1})-\int_{\Gamma^m}\gamma(\vec n^{m})ds+\sigma(x_r^{m}-x_l^{m})\nonumber\\
&\leq\Bigl(G(\theta^m)\partial_s \vec X^{m+1},~\partial_s\vec X^{m+1}-\partial_s\vec X^m\Bigr)_{\Gamma^m}^h+\int_{\Gamma^m}\gamma(\vec n^m)ds\nonumber\\
&\qquad-\sigma(x_r^{m+1}-x_l^{m+1})-\int_{\Gamma^m}\gamma(\vec n^{m})ds+\sigma(x_r^{m}-x_l^{m})\nonumber\\
&=\Bigl(G(\theta^m)\partial_s \vec X^{m+1},~\partial_s\vec X^{m+1}-\partial_s\vec X^m\Bigr)_{\Gamma^m}^h- \;\sigma\,\Bigl[(x_r^{m+1}-x_r^m) - (x_l^{m+1}-x_l^m)\Bigr]\nonumber\\
&=-\tau\Bigl(\partial_s\mu^{m+1},
~\partial_s\mu^{m+1}\Bigr)_{\Gamma^m}^h-\frac{1}{\eta}
\Bigr[\frac{(x_l^{m+1}-x_l^m)^2}{\tau}\, + \frac{(x_r^{m+1}-x_r^m)^2}{\tau}\Bigr]\nonumber\\
&\leq 0,\qquad m\ge0,
\end{align}
which immediately implies the energy dissipation \eqref{edppopen45}.
\end{proof}

\begin{remark}
All the results in Section 4 on energy dissipation of the ES-PFEM
\eqref{eqn:full discretization gamma(theta) torus} for motion of a closed curve can be extended to the ES-PFEM
\eqref{eqn:full discretization gamma(theta) open} for the motion of an open curve in solid-state dewetting. Again, the details are
omitted here for brevity.
\end{remark}

\section{Numerical results}

In this section, we report numerical results of the performance of our
proposed ES-PFEM  \eqref{eqn:full discretization gamma(theta) torus}
and \eqref{eqn:full discretization gamma(theta) open}
for the evolution of a closed curve and an open curve, respectively.
We will test their spatial/temporal convergent rates and energy dissipation, and investigate their area/mass loss and
mesh quality during the evolution.

To measure the difference between two curves $\Gamma_1$ and $\Gamma_2$, we adopt the manifold distance $M(\Gamma_1,\Gamma_2)$ which was introduced
in \citep{bao2020energy}. When $\Gamma_1$ and $\Gamma_2$ are two closed curves, let $\Omega_1$ and $\Omega_2$ be the regions
enclosed by $\Gamma_1$ and $\Gamma_2$, respectively; and
when they are two open curves above
the flat substrate, let $\Omega_1$ and $\Omega_2$ be the regions
enclosed between the flat substrate and $\Gamma_1$ and $\Gamma_2$, respectively. The manifold distance $M(\Gamma_1,\Gamma_2)$   is defined as
\citep{bao2020energy} (cf. Figure \ref{fig:illustration of manifold distance}):
\begin{equation}\label{manifold distance}
M(\Gamma_1,\Gamma_2):=|(\Omega_1\backslash \Omega_2)\cup (\Omega_2\backslash \Omega_1)|=|\Omega_1|+|\Omega_2|-2|\Omega_1\cap \Omega_2|,
\end{equation}
\begin{figure}[http]
\centering
\includegraphics[width=0.8\textwidth]{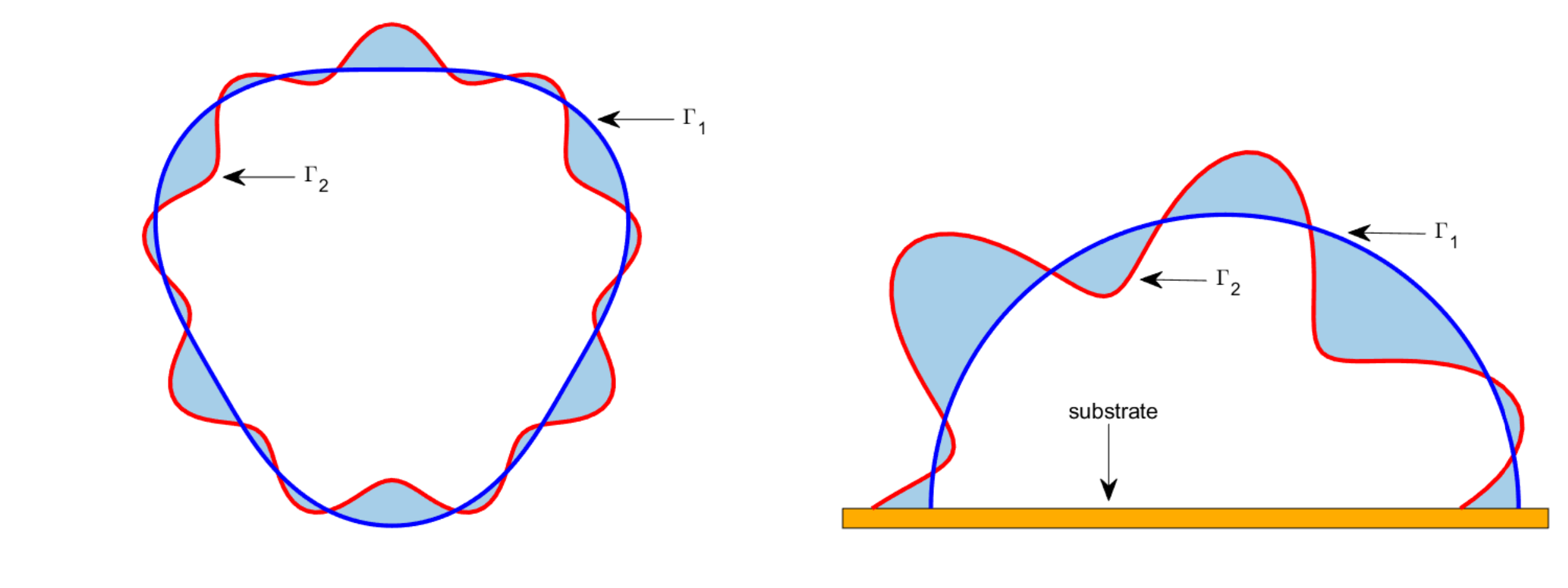}
\caption{An illustration of the manifold distance $M(\Gamma_1,\Gamma_2)$. Two curves $\Gamma_1$ and $\Gamma_2$ are colored by blue and red, respectively, the substrate is colord by brown and $M(\Gamma_1,\Gamma_2)$ is defined as the area of the light blue region.}
\label{fig:illustration of manifold distance}
\end{figure}
\noindent where $|\Omega|$ denotes the area of $\Omega$.

Suppose $\Gamma^m$ is the numerical approximation of $\Gamma(t=t_m=m\tau)$
with  mesh size $h$ and time step $\tau$ under the choice of $\tau=h^2$,
for simplicity, since formally our ES-PFEM is first order accurate in time and second order accurate in space.
The numerical error is defined as
\begin{equation}
e^{h}(t_m):=M(\Gamma^m,\Gamma(t=t_m)),\qquad m\ge0,
\end{equation}
where $\Gamma(t=t_m)$ is obtained numerically with a very small mesh size
$h=h_e$ and a very small time step $\tau=\tau_e$, e.g.
$h_e=2^{-8}$ and $\tau_e=2^{-16}$, in practical computations when
the exact solution is not available.
Let $A^h(t=t_m)$ be the area/mass of the region enclosed by
$\Gamma^m$ if it is a closed curve, and respectively,
the region between the flat substrate and
$\Gamma^m$ if it is an open curve.
Then the normlized area/mass loss $\frac{\Delta A^h(t_m)}{A^h(0)}$ and the mesh ratio $R^h(t=t_m)$ which
is used to measure the mesh quality of $\Gamma^m$,
are defined as
\begin{equation}
\frac{\Delta A^h(t_m)}{A^h(0)}:=\frac{A^h(t=t_m)-A^h(0)}{A^h(0)},\qquad
R^h(t=t_m):=\frac{h_{\rm max}^m}{h_{\rm min}^m},
\qquad m\ge0,
\end{equation}
where
\[
h_{\rm max}^m:=\max_{1\le j\le N}\ |\vec h_j^m|,\qquad
h_{\rm min}^m:=\min_{1\le j\le N}\ |\vec h_j^m|,\qquad m\ge0.
\]


In the following numerical simulations, the initial shapes are taken as a $4\times 1$ rectangle for both closed curves and open curves except that they are stated otherwise. For solid-state dewetting problems, we always choose the contact line mobility
$\eta=100$ in \eqref{eqn:weakBC2 gamma(theta) open} \citep{bao2020energy}.

\subsection{Results of the ES-PFEM
\eqref{eqn:full discretization gamma(theta) torus}
for the evolution of closed curves}

Figure \ref{fig:convergence rate test1} plots spatial convergence
rate of the ES-PFEM \eqref{eqn:full discretization gamma(theta) torus}
with the anisotropic surface energy $\gamma(\theta)=1+\beta\cos(4\theta)$
for different times $t$ and $\beta$. Figure
\ref{fig:area conservation closed} depicts the time evolution of
the normalized area/mass loss $\frac{\Delta A^h(t_m)}{A^h(0)}$ and the energy dissipation
$W_c^h(t_m)$ with $\gamma(\theta)=1+0.05\cos(4\theta)$
 for different $h$. Finally Figure \ref{fig:mesh ratio closed} shows
the time evolution of the mesh ratio $R^h(t=t_m)$ with
$\gamma(\theta)\equiv 1$ and $\gamma(\theta)=1+0.05\cos(4\theta)$
 for different $h$.


\begin{figure}[t!]
\centering
\includegraphics[width=0.8\textwidth]{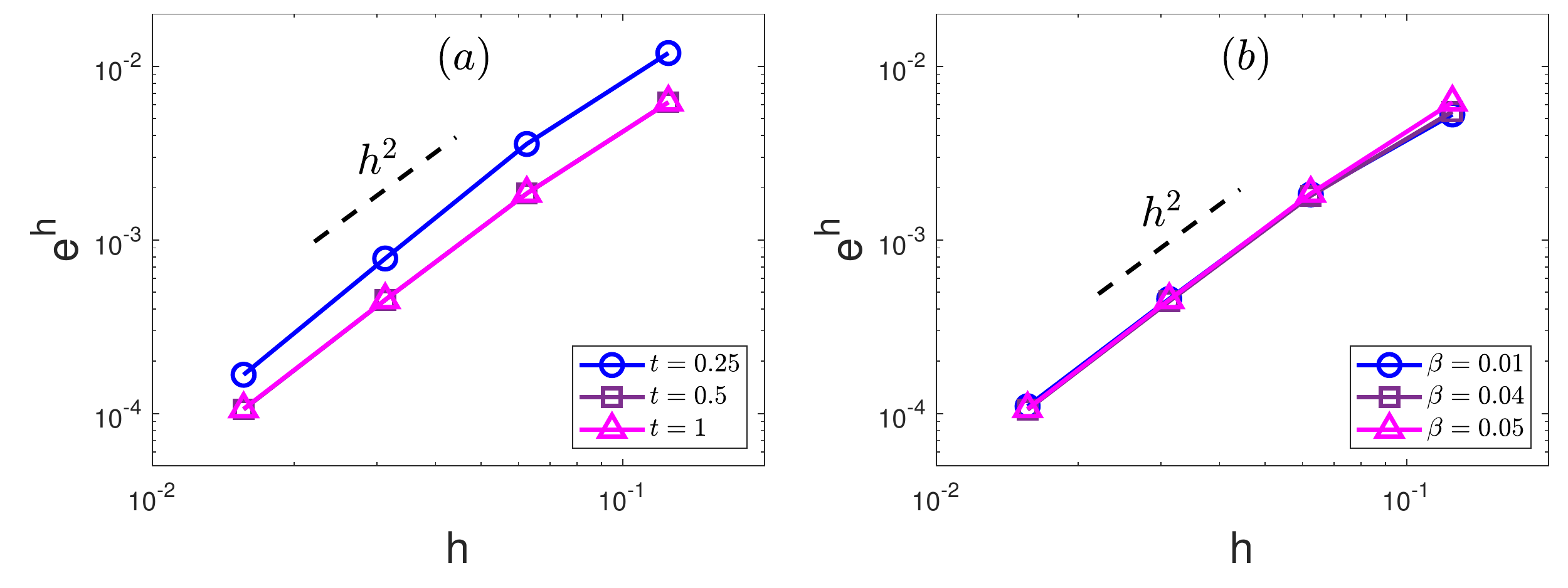}
\caption{Plots of $e^h$ vs $h$ to test the spatial convergence rate
of the ES-PFEM \eqref{eqn:full discretization gamma(theta) torus}
with  $\gamma(\theta)=1+\beta\cos(4\theta)$ for: (a) different times
at $t = 0.25,\,t=0.5$ and $t=1$ with $\beta=0.05$; and (b) different anisotropic strengths $\beta=0.01,\, \beta=0.04$ and $\beta=0.05$ at time $t=2$.}
\label{fig:convergence rate test1}
\end{figure}



\begin{figure}[t!]
\centering
\includegraphics[width=0.8\textwidth]{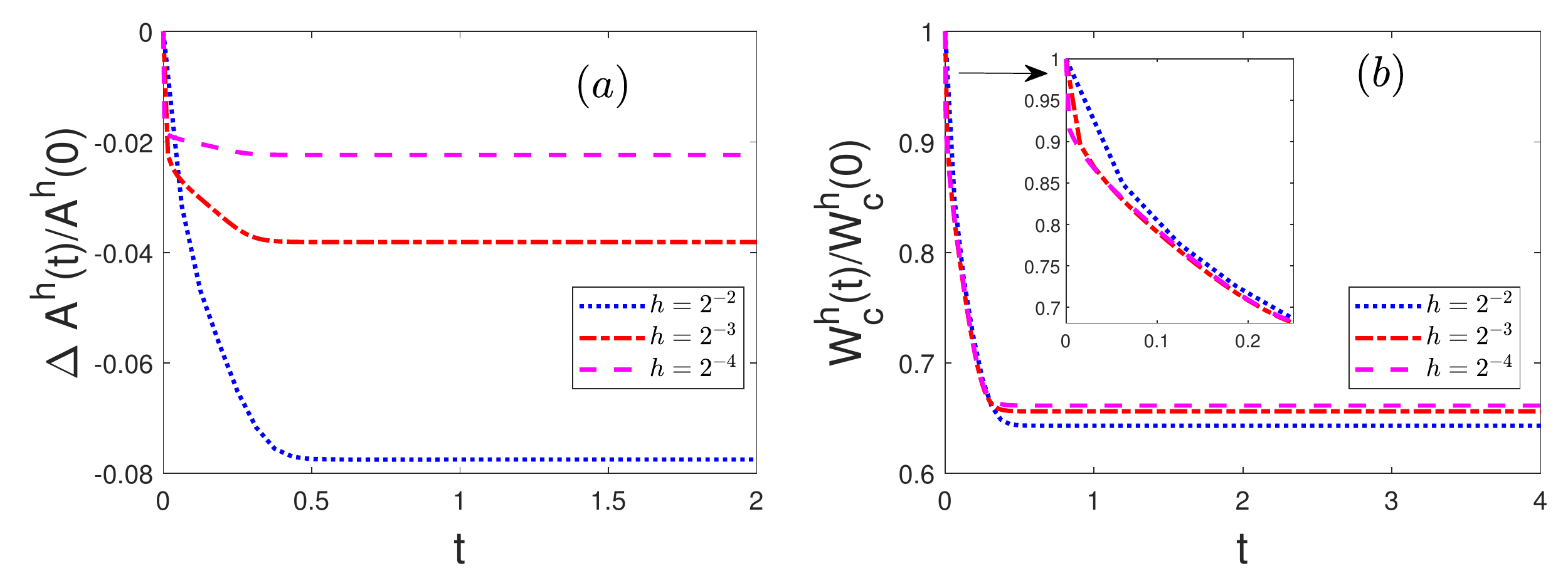}
\caption{Time evolution of the normalized area/mass loss $\frac{\Delta A^h(t)}{A^h(0)}$ (left (a)) and the normalized energy $\frac{W^h_c(t)}{W^h_c(0)}$ (right (b)) for different
mesh sizes $h$. The anisotropic surface energy is chosen as $\gamma=1+0.05\cos4\theta$.}
\label{fig:area conservation closed}
\end{figure}

\begin{figure}[t!]
\centering
\includegraphics[width=0.8\textwidth]{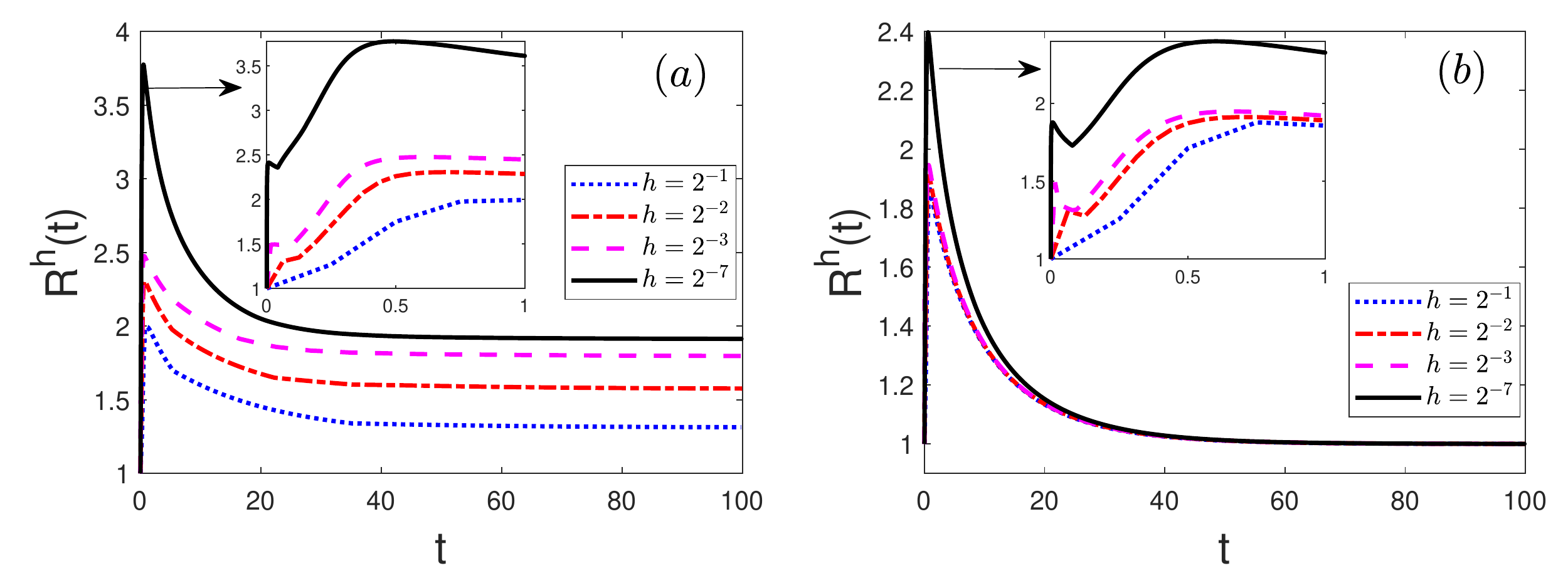}
\caption{Time evolution of the mesh ratio $R^h(t)$ for different mesh sizes $h$ with: (a) an anisotropic surface energy $\gamma(\theta)=1+0.05\cos 4\theta$; and (b) an isotropic surface energy $\gamma(\theta)\equiv 1$.  }
\label{fig:mesh ratio closed}
\end{figure}

From Figures \ref{fig:convergence rate test1}-\ref{fig:mesh ratio closed},
we can draw the following conclusions for the ES-PFEM
\eqref{eqn:full discretization gamma(theta) torus}
 for the evolution
of closed curves under anisotropic surface diffusion:

(i) The ES-PFEM \eqref{eqn:full discretization gamma(theta) torus}
is second order accurate in space and first order accurate in time (cf. Figure
\ref{fig:convergence rate test1}).

(ii) It is unconditionally energy stable when the anisotropic surface energy
$\gamma(\theta)$ satisfies those energy dissipation conditions in Section 3
(cf. Figure \ref{fig:area conservation closed}b).

(iii) The mesh ratio $R^h(t=t_m)$ increases during a short period near $t=0$ and then it decreases to a constant when $t\gg1$. For isotropic surface energy, i.e. isotropic surface diffusion, $R^h(t=t_m)\to 1$ when
$t\to+\infty$ (cf. Figure \ref{fig:mesh ratio closed}b), which indicates
asymptotic mesh equal distribution (AMED) of the ES-PFEM
\eqref{eqn:full discretization gamma(theta) torus}
for isotropic surface diffusion. On the other hand,
for anisotropic surface energy, i.e. anisotropic surface diffusion, $R^h(t=t_m)\to C>1$ when
$t\to+\infty$ (cf. Figure \ref{fig:mesh ratio closed}a), which indicates
asymptotic mesh quasi-equal distribution (AMQD) of the ES-PFEM
\eqref{eqn:full discretization gamma(theta) torus}
for anisotropic surface diffusion.

(iv) Area/mass loss is observed during a short period near $t=0$,
especially when the mesh size $h$ is not small
(cf. Figure \ref{fig:area conservation closed}a). When
$t=t_m\gg1$, area/mass is almost conserved and
we observed numerically that $|\frac{\Delta A^h(t_m)}{A^h(0)}|\le Ch^2$, i.e.
it converges quadratically and this agrees with the second order accuracy
in space of the ES-PFEM \eqref{eqn:full discretization gamma(theta) torus}.


\begin{figure}[t!]
\centering
\includegraphics[width=0.8\textwidth]{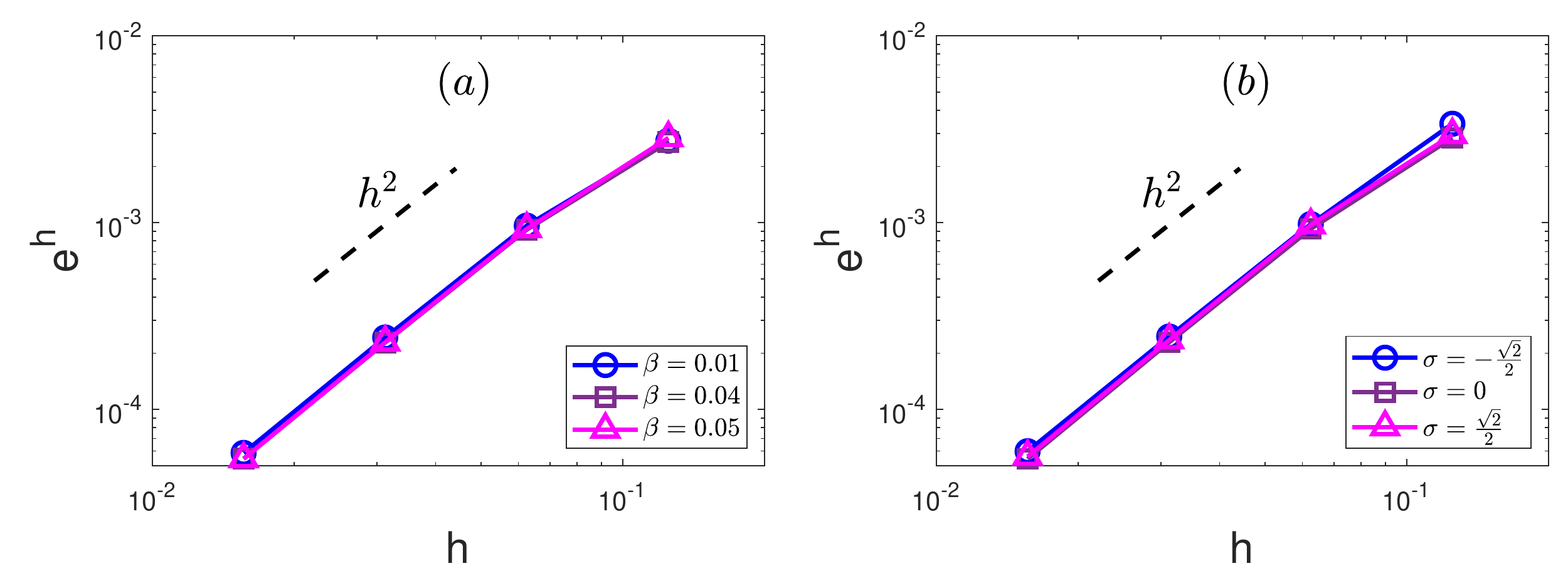}
\caption{Plots of $e^h(t=2)$ vs $h$ to test the spatial convergence rate
of the ES-PFEM \eqref{eqn:full discretization gamma(theta) open}
with  $\gamma(\theta)=1+\beta\cos(4\theta)$ for: (a) different $\beta$ with $\sigma=0$ in \eqref{def of f}; (b) different $\sigma$ in \eqref{def of f} with $\beta=0.05$.}
\label{fig:convergence rate test2}
\end{figure}

\subsection{Results of the ES-PFEM
\eqref{eqn:full discretization gamma(theta) open}
for the evolution of open curves}

Figure \ref{fig:convergence rate test2} plots spatial convergence
rate of the ES-PFEM \eqref{eqn:full discretization gamma(theta) open}
with the anisotropic surface energy $\gamma(\theta)=1+\beta\cos(4\theta)$
for different times $t$ and $\beta$. Figure \ref{fig:area conservation open}
 depicts the time evolution of
the normalized area/mass loss $\frac{\Delta A^h(t_m)}{A^h(0)}$ and the energy dissipation
$W_o^h(t_m)$ with $\gamma(\theta)=1+0.05\cos(4\theta)$
 for different $h$. Finally Figure \ref{fig:mesh ratio open} shows
the time evolution of the mesh ratio $R^h(t=t_m)$ with
$\gamma(\theta)\equiv 1$ and $\gamma(\theta)=1+0.05\cos(4\theta)$
 for different $h$.



\begin{figure}[t!]
\centering
\includegraphics[width=0.8\textwidth]{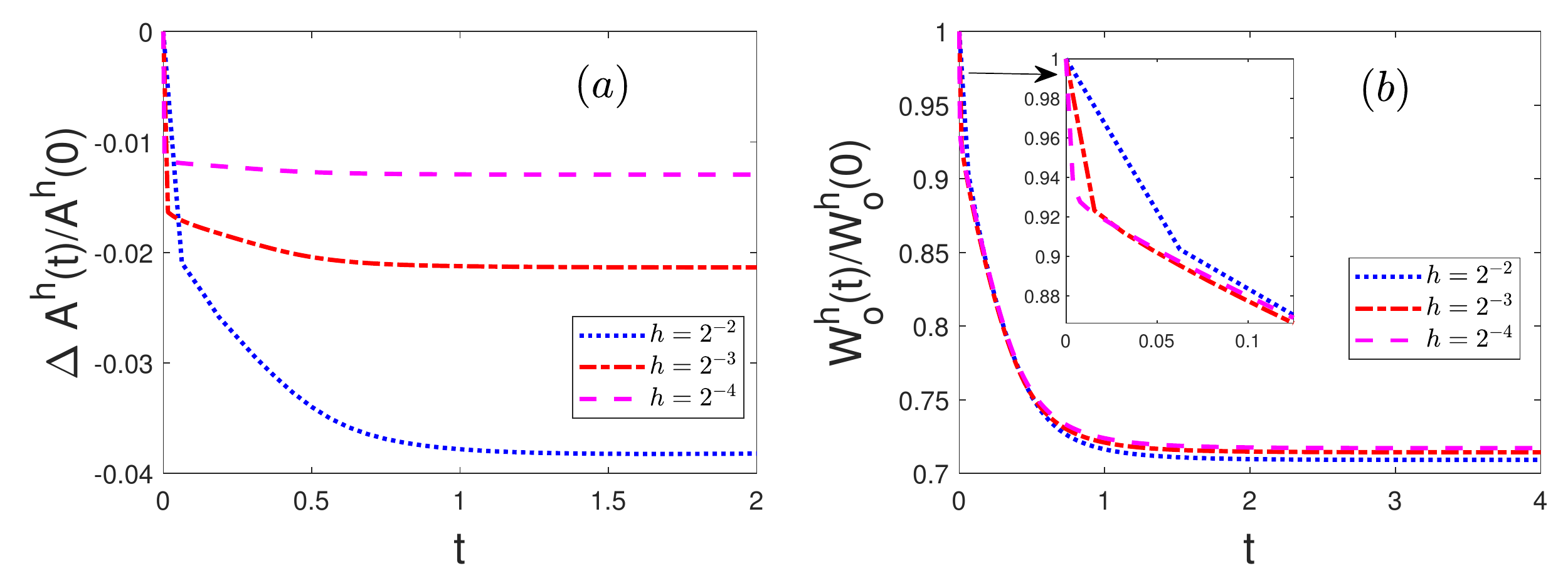}
\caption{Time evolution of the normalized area/mass loss $\frac{\Delta A^h(t)}{A^h(0)}$ (left (a)) and the normalized energy $\frac{W^h_o(t)}{W^h_o(0)}$ (right (b)) for different
mesh sizes $h$. The anisotropic surface energy is chosen as $\gamma=1+0.05\cos4\theta$ and material constant $\sigma=-\frac{\sqrt{2}}{2}$ in
\eqref{def of f}.}
\label{fig:area conservation open}
\end{figure}


\begin{figure}[t!]
\centering
\includegraphics[width=0.8\textwidth]{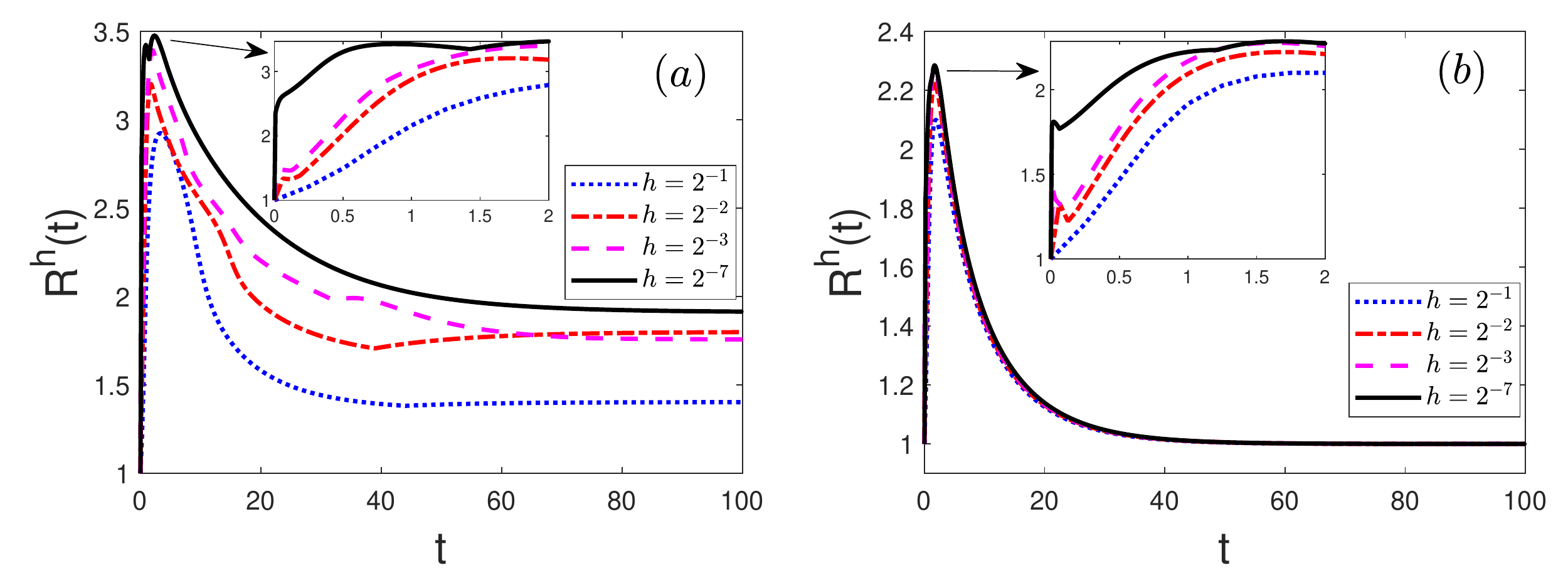}
\caption{Time evolution of the mesh ratio $R^h(t)$ for different mesh sizes $h$ with $\sigma=-\frac{\sqrt{2}}{2}$ in
\eqref{def of f} and: (a) an anisotropic surface energy $\gamma(\theta)=1+0.05\cos 4\theta$; and (b) an isotropic surface energy $\gamma(\theta)\equiv 1$.}
\label{fig:mesh ratio open}
\end{figure}

Again, from Figures
\ref{fig:convergence rate test2}-\ref{fig:mesh ratio open},
we can draw the following conclusions for the ES-PFEM
\eqref{eqn:full discretization gamma(theta) open}
for the evolution
of open curves under anisotropic surface diffusion with applications
in solid-state dewetting:

(i) The ES-PFEM
\eqref{eqn:full discretization gamma(theta) open}
 is second order accurate in space and first order accurate in time (cf. Figure
\ref{fig:convergence rate test2}).

(ii) It is unconditionally energy stable when the anisotropic surface energy
$\gamma(\theta)$ satisfies those energy dissipation conditions in Section 3
(cf. Figure \ref{fig:area conservation open}b).

(iii) The mesh ratio $R^h(t=t_m)$ increases during a short period near $t=0$ and then it decreases to a constant when $t\gg1$. For isotropic surface energy, i.e. isotropic surface diffusion, $R^h(t=t_m)\to 1$ when
$t\to+\infty$ (cf. Figure \ref{fig:mesh ratio open}b), which indicates
asymptotic mesh equal distribution (AMED) of the ES-PFEM
\eqref{eqn:full discretization gamma(theta) open} for isotropic surface diffusion. On the other hand,
for anisotropic surface energy, i.e. anisotropic surface diffusion, $R^h(t=t_m)\to C>1$ when
$t\to+\infty$ (cf. Figure \ref{fig:mesh ratio open}a), which indicates
asymptotic mesh quasi-equal distribution (AMQD) of the ES-PFEM
\eqref{eqn:full discretization gamma(theta) open}
for anisotropic surface diffusion.

(iv) Area/mass loss is observed during a short period near $t=0$,
especially when the mesh size $h$ is not small
(cf. Figure \ref{fig:area conservation open}a). When
$t=t_m\gg1$, area/mass is almost conserved and
we observed numerically that $|\frac{\Delta A^h(t_m)}{A^h(0)}|\le Ch^2$, i.e.
it converges quadratically and this agrees with the second order accuracy
in space of the ES-PFEM \eqref{eqn:full discretization gamma(theta) open}.

\begin{figure}[t!]
\centering
\includegraphics[width=0.8\textwidth]{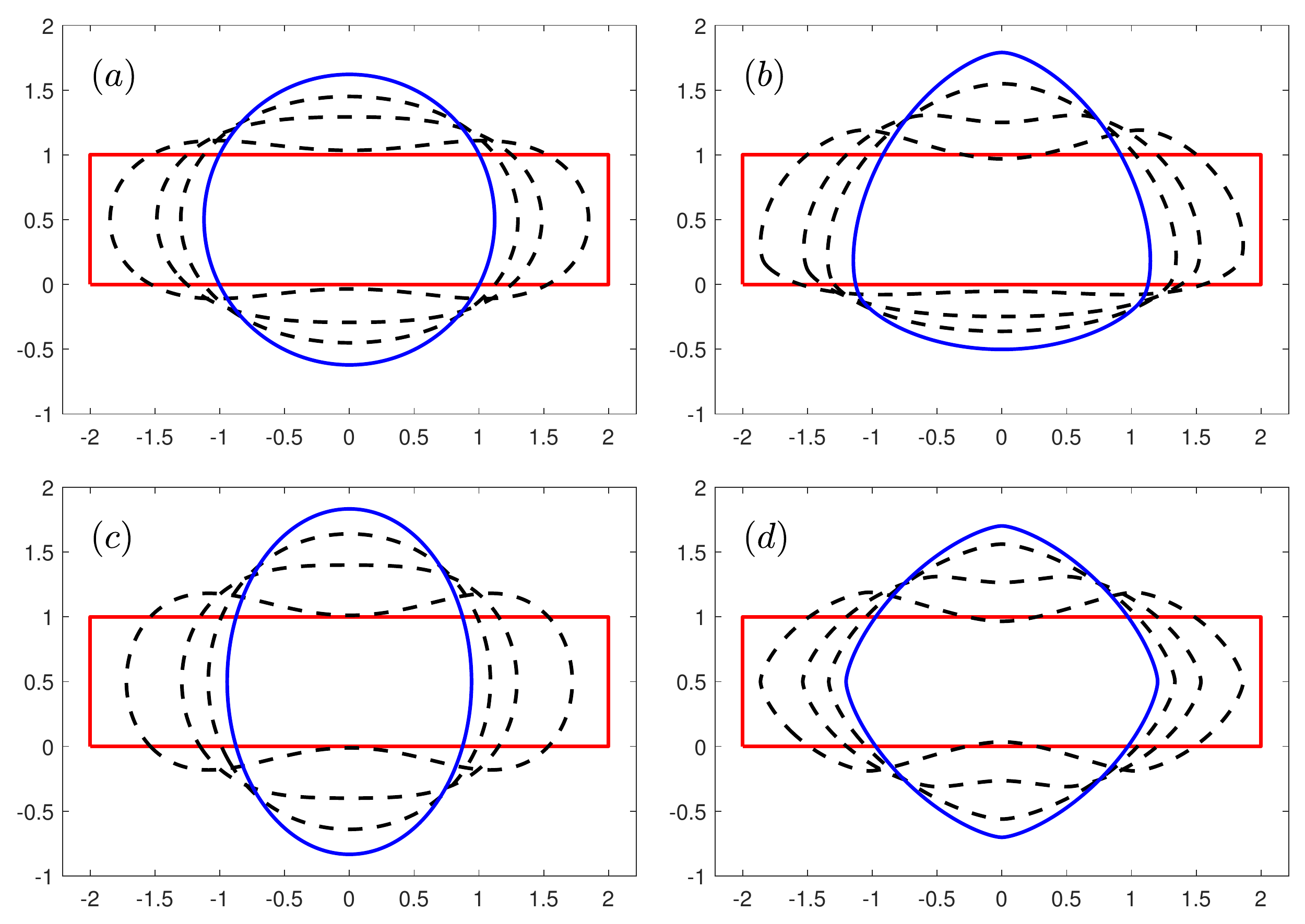}
\caption{Morphological evolutions of a close rectangular curve under
anisotropic surface diffusion with different anisotropic surface energies:  (a) $\gamma(\theta)\equiv 1$, (b) $\gamma(\theta)=1+\frac{1}{10}\cos 3\theta$, (c) $\gamma(\theta)=\sqrt{1+\cos^2\theta}$, and (d) $\gamma(\theta)=1+\frac{1}{17}\cos 4\theta$. Other parameters are chosen as $h=2^{-6}, \tau=h^2$. The red line is the initial shape, the black dashed lines are some snapshots during the evolution and the blue line is the equilibrium shape. }
\label{fig:shape closed}
\end{figure}

\begin{figure}[t!]
\centering
\includegraphics[width=0.8\textwidth]{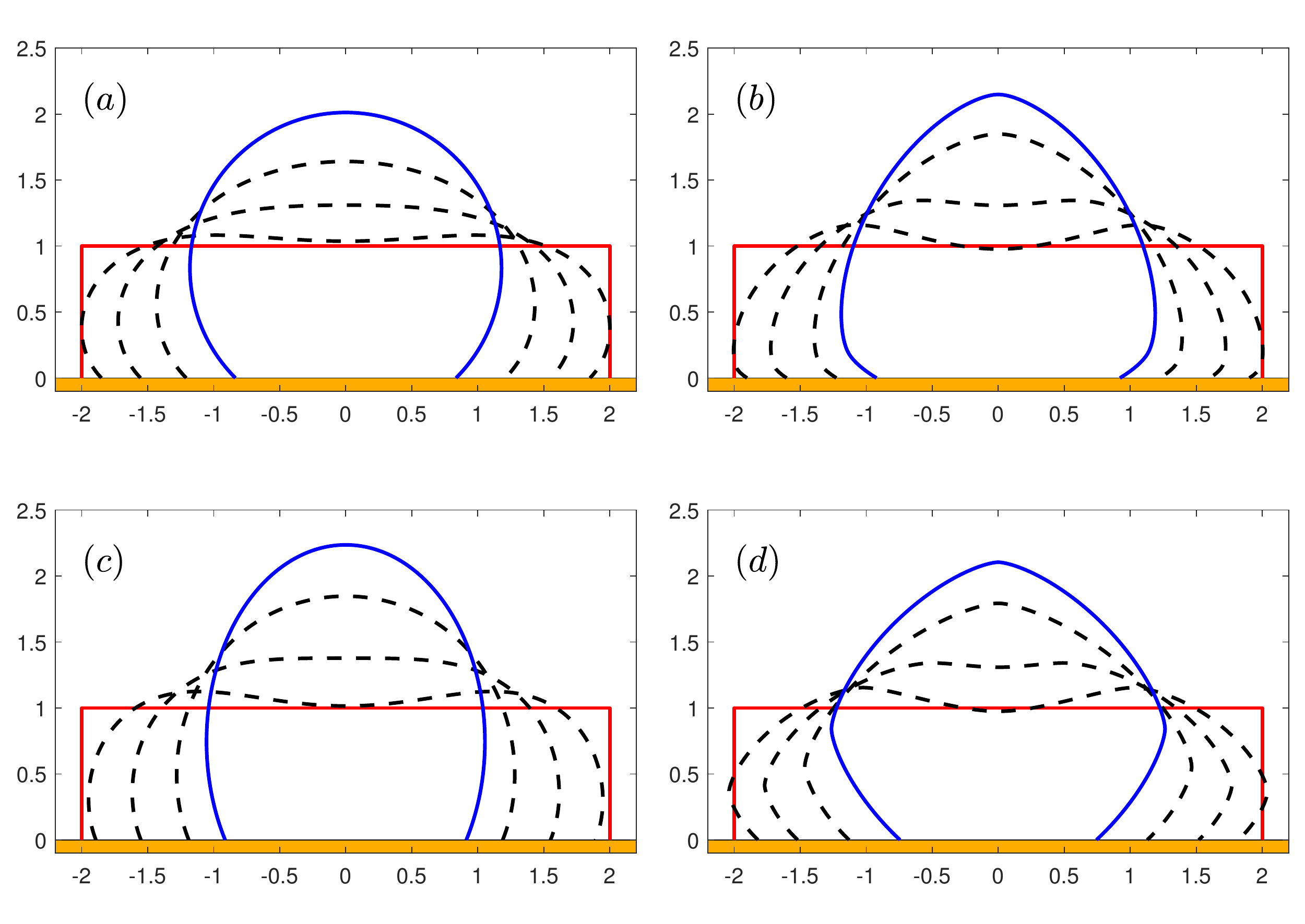}
\caption{Morphological evolutions of an open rectangular curve under anisotropic surface diffusion with different anisotropic surface energies:  (a) $\gamma(\theta)\equiv 1$, (b) $\gamma(\theta)=1+\frac{1}{10}\cos 3\theta$, (c) $\gamma(\theta)=\sqrt{1+\cos^2\theta}$, and (d) $\gamma(\theta)=1+\frac{1}{17}\cos 4\theta$. Other parameters are chosen as $\sigma=-\frac{\sqrt{2}}{2}, h=2^{-6}, \tau=h^2$. The red line is the initial shape, the black dashed lines are some snapshots during the evolution, the blue line is the equilibrium shape and the brown base represents the substrate. }
\label{fig:shape open}
\end{figure}

\subsection{Applications of the ES-PFEM for morphological evolution}

Finally we examine the morphological evolution under different anisotropic surface energies by our proposed ES-PFEM. The morphological evolutions of closed curves and open curves from a $4\times 1$ rectangle  towards their equilibrium shapes are shown in Figure \ref{fig:shape closed} and Figure \ref{fig:shape open}, respectively.
Four different anisotropic surface energies are taken as the isotropic energy $\gamma(\theta)\equiv 1$, the $k$-fold anisotropic energies $\gamma(\theta)=1+\frac{1}{1+3^2}\cos(3\theta)=1+\frac{1}{10}\cos (3\theta)$, $\gamma(\theta)=1+\frac{1}{1+4^2}\cos(4\theta)=1+\frac{1}{17}\cos (4\theta)$, and the ellipsoidal anisotropic energy $\gamma(\theta)=\sqrt{1+\cos^2\theta}$. For open curves,
we take $\sigma=-\frac{\sqrt{2}}{2}$ in \eqref{def of f}. From Corollaries \ref{coro: ellipsoidal} and \ref{coro: k-fold}, the parameters $a=1, \,b=1$ in $\gamma(\theta)=\sqrt{1+\cos^2\theta}$ attain the largest ratio $\frac{b}{a}=1$ that we have proved for the ellipsoidal anisotropy, and the parameters $\beta=\frac{1}{1+3^2}, \beta=\frac{1}{1+4^2}$ are also the largest $\beta_{\max}=\frac{1}{1+k^2}$ for the $k$-fold anisotropy.

\begin{figure}[t!]
\centering
\includegraphics[width=0.7\textwidth]{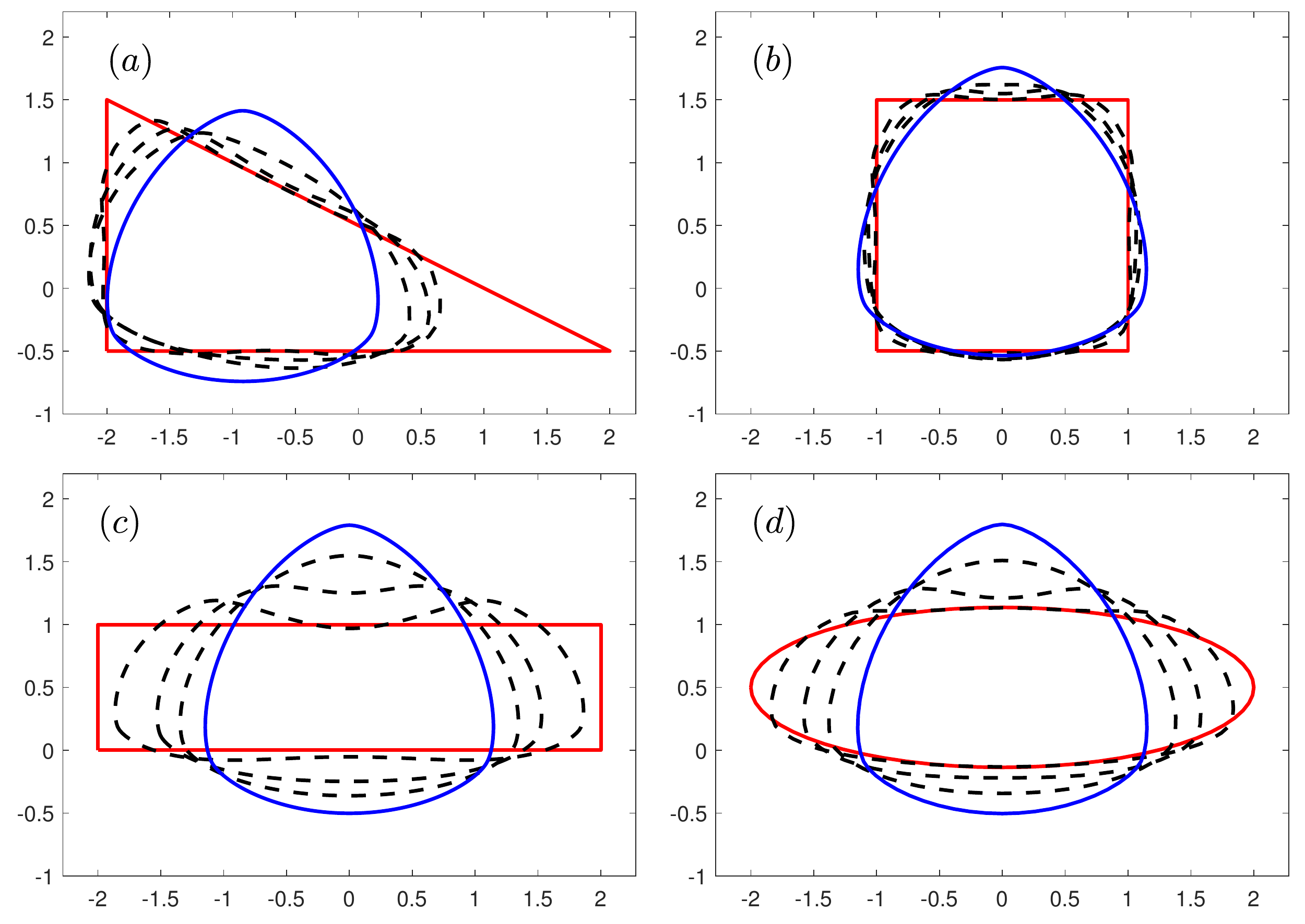}
\caption{Morphological evolutions of different closed initial curves under anisotropic surface diffusion with anisotropic surface energy $\gamma(\theta)=1+\frac{1}{10}\cos(3\theta)$:  (a) an initial $4\times 2$ right triangle, (b) an initial $2\times 2$ square, (c) an initial $4\times 1$ rectangle, and (d) an initial ellipse with length $4$ and width $\frac{4}{\pi}$. Other parameters are chosen as $h=2^{-6}, \tau=h^2$. The red line is the initial shape, the black dashed lines are some snapshots during the evolution, the blue line is the equilibrium shape. }
\label{fig:diff initial}
\end{figure}

\begin{figure}[t!]
\centering
\includegraphics[width=0.7\textwidth]{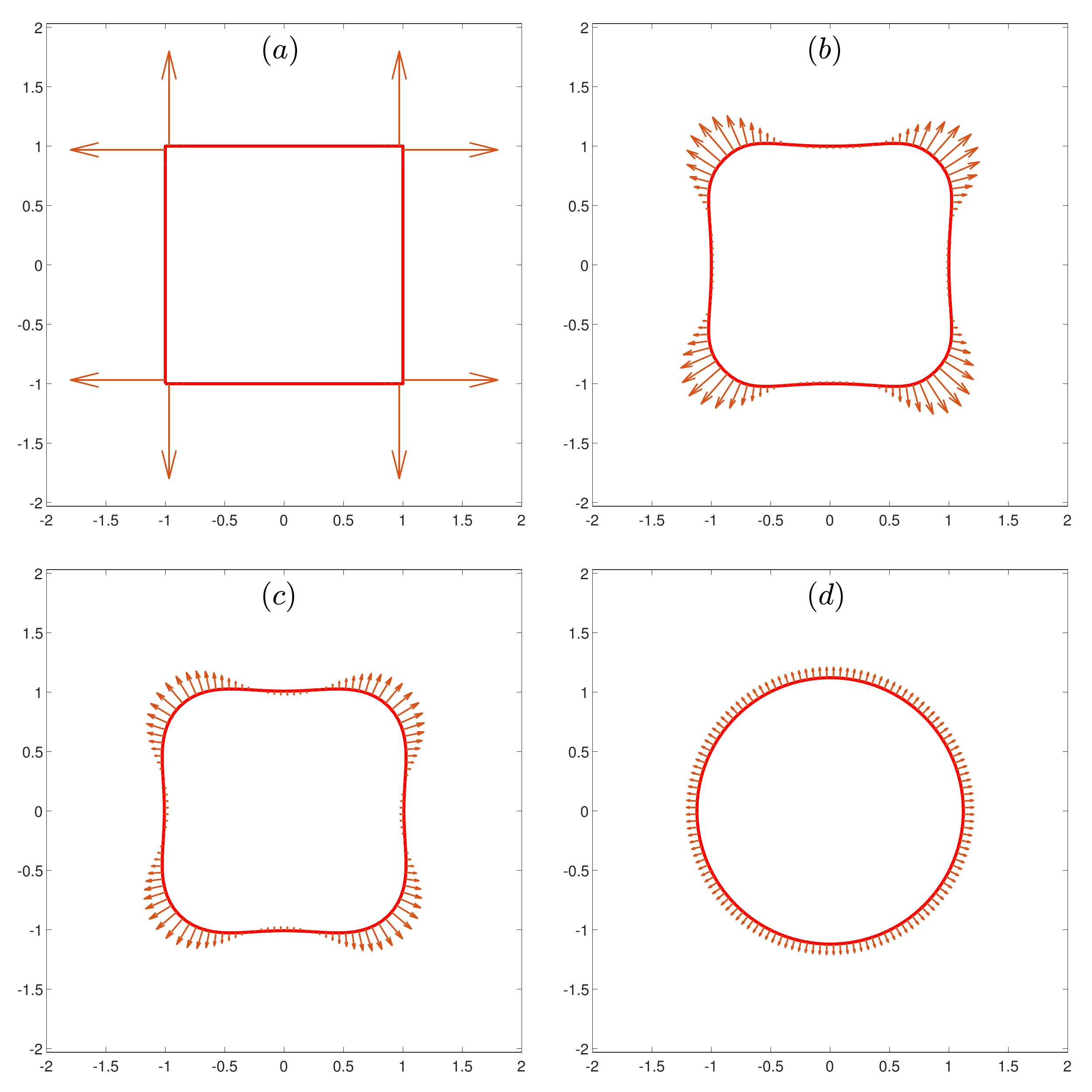}
\caption{Time evolution of the curvature $\kappa$ under isotropic surface diffusion at: (a) $t=0$; (b) $t=\tau$; (c) $t=3 \tau$; (d) $t=10^4\tau$, where mesh size $h=2^{-5}$ and time step $\tau=h^2$. The initial curve is a $2\times 2$ square. The length of the arrow in the figure is scaled as one-tenth of the actual $\kappa$.}
\label{fig:evolution of weighted mean curvature}
\end{figure}

\begin{figure}[t!]
\centering
\includegraphics[width=0.7\textwidth]{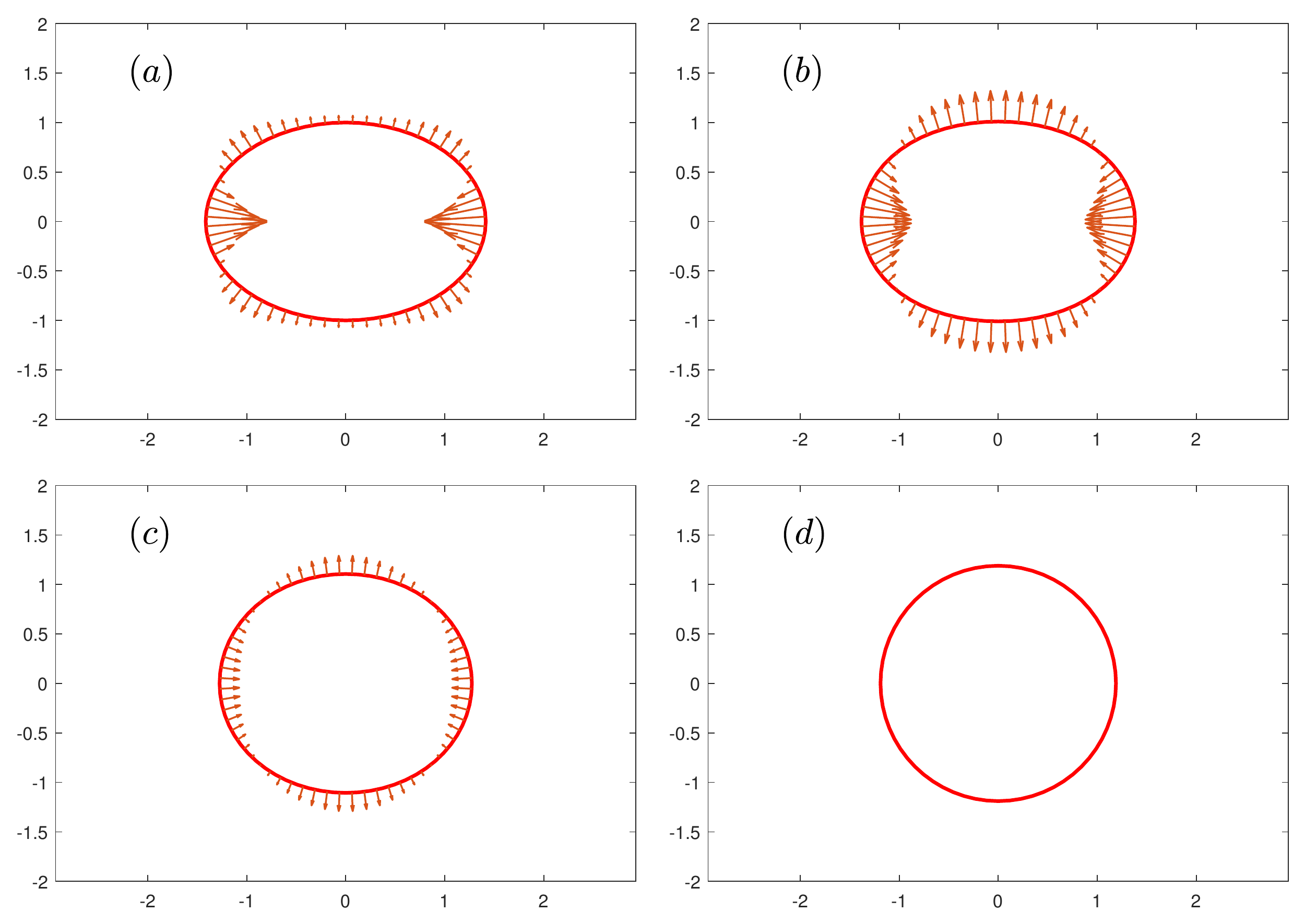}
\caption{Time evolution of the normal velocity $V$ under isotropic surface diffusion at: (a) $t=0$; (b) $t=20\tau$; (c) $t=200 \tau$; (d) $t=10^4\tau$, where mesh size $h=2^{-5}$ and time step $\tau=h^2$. The initial curve is an ellipse with length $2\sqrt{3}$ and width $2$. The length of the arrow in the figure is scaled as one-tenth of the actual $V$ for $t>0$, and as one-twentieth of the actual $V$ for $t=0$.}
\label{fig:evolution of V}
\end{figure}

As observed from Fig. \ref{fig:shape closed}(a)-(d) and Fig. \ref{fig:shape open}(a)-(d), the equilibrium shapes for the isotropic surface energy and the ellipsoidal anisotropic surface energy are indeed circle and ellipsis, respectively. As for $k$-fold anisotropy, when $k$ is changed from $3$ to $4$, the number of edges in their equilibrium shapes are also changed accordingly, as expected, which agree with both theoretical predictions and previous numerical results. Moreover, our ES-PFEM can handle the largest ratio $\frac{b}{a}=1$ and the largest $\beta=\beta_{\max}$ well for both closed curves and open curves.

Our ES-PFEM also works well for different initial shapes including continuous but piecewise smooth initial curves. Figure \ref{fig:diff initial} plots the morphological evolutions of four different closed initial configurations with $k$-fold anisotropy $\gamma(\theta)=1+\frac{1}{10}\cos(3\theta)$. We can see our ES-PFEM can handle successfully different initial curves with the same area, and the final equilibrium  of different initial configurations is the same, which is consistent with the theoretical result.

Another three important quantities in morphological evolutions are the weighted curvature $\mu^m$, the curvature $\kappa^m$, and the normal velocity $V^{m}:=\vec n^m\cdot \frac{\vec X^{m+1}-\vec X^m}{\tau}$ which is an numerical approximation of $V(\cdot,t=t_m)=\partial_{ss}\mu(\cdot,t=t_m)=\left.\vec n\cdot \partial_t \vec X\right|_{t=t_m}$. Notice that in our ES-PFEMs
\eqref{eqn:full discretization gamma(theta) torus} and \eqref{eqn:full discretization gamma(theta) open} for the evolution of a closed and open curve, respectively, we state how to compute numerically
$\mu^m$ for $m\ge1$, but
we do not show how to compute numerically $\mu^0$ and $\kappa^m$ for $m\ge0$.
In fact, for a given closed initial configuration $\Gamma^0=\vec X^{0}$ which might be continuous but only piecewise smooth such as a rectangle, one can adapt the following variational formulation to compute numerically $\mu^0 \in\mathbb{K}^h_p$ and $\kappa^m\in\mathbb{K}^h_p$:
\begin{equation}
\begin{split}
&\Bigl(\mu^{0},\vec n^0\cdot\boldsymbol{\omega}^h\Bigr)_{\Gamma^0}^h-\Bigl(G(\theta^0)\partial_s \vec X^{0},~\partial_s\boldsymbol{\omega}^h\Bigr)_{\Gamma^0}^h=0,
\quad\forall\boldsymbol{\omega}^h\in\mathbb{X}^h_p,\\
&\Bigl(\kappa^{m},\vec n^m\cdot\boldsymbol{\omega}^h\Bigr)_{\Gamma^m}^h-\Bigl(\partial_s \vec X^{m},~\partial_s\boldsymbol{\omega}^h\Bigr)_{\Gamma^m}^h=0,
\quad\forall\boldsymbol{\omega}^h\in\mathbb{X}^h_p, \qquad m\ge0.
\end{split}
\end{equation}
Similarly, for a given open initial configuration $\Gamma^0=\vec X^{0}$, one needs to replace $\mathbb{K}_p^h$ and $\mathbb{X}^h_p$ by $\mathbb{K}^h$ and $\mathbb{X}^h$, respectively. Figure
\ref{fig:evolution of weighted mean curvature} displays time evolution of the curvature $\kappa(t=t_m)$  at different times under isotropic surface diffusion starting from an initial $2\times 2$ square. Similarly, Figure \ref{fig:evolution of V} shows time evolution of the normal velocity $V(t=t_m)$ at different times under isotropic surface diffusion starting from an initial ellipse with length $2\sqrt{3}$ and width $2$.

From Fig. \ref{fig:evolution of weighted mean curvature}, we can see that: (i) the curvature $\kappa$ at the four sharp corners are discontinuous at $t=0$, which are `numerically' significant larger than those values of their neighbors (cf. Fig. \ref{fig:evolution of weighted mean curvature}a), (ii) after evolution of a few time steps with a small time step
size $\tau$, the sharp corners are being smoothed and the values of the curvature $\kappa$ become comparable with those values of their neighbours
(cf. Fig. \ref{fig:evolution of weighted mean curvature}b\&c), and (iii) when the curve reaches its equilibrium shape, the curvature $\kappa$ are almost the same at each point of the curve (cf. Fig. \ref{fig:evolution of weighted mean curvature}d).

\section{Conclusions}

By introducing a positive definite surface energy (density) matrix $G(\theta)$ depending on the anisotropic surface energy $\gamma(\theta)$,
we obtained new and simple variational formulations for the motion
of closed curves under anisotropic surface diffusion or open curves
under anisotropic surface diffusion and contact line migration with applications in solid-state dewetting in materials science.
We proved area/mass conservation and energy dissipation of the variational
problems. The variational problems were first discretized in space
by the parametric finite element method (PFEM) and then were
discretized in time by an implicit/expicit (IMEX) backward Euler method.
The full-discretization is semi-implicit and efficient since
only a linear system needs to be solved at each time step.
We identified different energy dissipation conditions on the
anisotropic surface energy $\gamma(\theta)$ such that both the semi-discretization and full-discretization are unconditionally
energy stable. Our numerical results suggested that the proposed
energy-stable PFEM (ES-PFEM) has nice mesh quality -- asymptotic
mesh quasi-equal distribution -- of the curves
during their dynamics, i.e. no re-meshing is needed during the simulation.
In the future, we will extend the new variational formulation to anisotropic
surface diffusion in three dimensions \cite{Jiang,Zhao}
and other geometric flows arising
from different applications.


\bigskip

\begin{appendices}
\fakesection{{\bf Appendix A}. Two trigonometric identities and their proof}
\begin{center}
{\bf Appendix A}. Two trigonometric identities and their proof
\end{center}
\setcounter{equation}{0}
\renewcommand{\theequation}{A.\arabic{equation}}

Here we show two trigonometric identities which are used to prove Lemma B.1
in Appendix B.

\begin{lem}\label{trignometric identity}
$\forall n\in\mathbb{Z}^+,\, \forall\theta,\,\phi\in [-\pi,\pi]$, the following two trigonometric identities hold:
\begin{equation}\label{sin}
\sin(n\theta)-\sin(n\phi)-n\cos(n\theta)\sin(\theta-\phi)
=(1-\cos(\theta-\phi))\left(n\sin(n\theta)+\sum_{l=1}^{n-1}
2l\sin\left(l\theta+(n-l)\phi\right)\right),
\end{equation}
\begin{equation}\label{cos}
\cos(n\theta)-\cos(n\phi)+n\sin(n\theta)\sin(\theta-\phi)
=(1-\cos(\theta-\phi))\left(n\cos(n\theta)+\sum_{l=1}^{n-1}
2l\cos\left(l\theta+(n-l)\phi\right)\right).
\end{equation}
\end{lem}

\begin{proof}
To prove \eqref{sin}, noticing the trigonometric identity
\begin{equation}\label{prod to sum} \cos\alpha\sin\beta=\frac{\sin(\alpha+\beta)-\sin(\alpha-\beta)}{2},
\end{equation}
subtracting the left hand side of \eqref{sin} by its right hand side, we get
\begin{align}
&\left(1-\cos(\theta-\phi)\right)\left(n\sin(n\theta)+
\sum_{l=1}^{n-1}2l\sin(l\theta+(n-l)\phi)\right)-
\left[\sin(n\theta)-\sin(n\phi)-n\cos(n\theta)
\sin(\theta-\phi)\right]\nonumber\\
&=\sum_{l=1}^{n-1}2l\sin(l\theta+(n-l)\phi)-
\sum_{l=1}^{n-1}2l\cos(\theta-\phi)\sin(l\theta+(n-l)\phi)+
n\sin(n\theta)-n\cos(\theta-\phi)\sin(n\theta)\nonumber\\
&\quad-\left[\sin(n\theta)-\sin(n\phi)-
n\frac{\sin\left((n+1)\theta-\phi\right)-\sin\left((n-1)
\theta+\phi\right)}{2}\right]\nonumber\\
&=\sum_{l=1}^{n-1}2l\sin(l\theta+(n-l)\phi)-
\sum_{l=1}^{n-1}l\left[\sin\Bigl((l+1)\theta+(n-l-1)\phi\Bigr)+
\sin\Bigl((l-1)\theta+(n-l+1)\phi\Bigr)\right]\nonumber\\
&\quad+n\sin(n\theta)-n\frac{\sin\left((n+1)\theta-\phi\right)+
\sin\left((n-1)\theta+\phi\right)}{2}\nonumber\\
&\quad-\left[\sin(n\theta)-\sin(n\phi)-n\frac{\sin\left((n+1)\theta-
\phi\right)-\sin\left((n-1)\theta+\phi\right)}{2}\right]\nonumber\\
&=\sum_{l=1}^{n-1}2l\sin(l\theta+(n-l)\phi)-\sum_{l=0}^{n-2}(l+1)
\sin(l\theta+(n-l)\phi)-\sum_{l=2}^{n}(l-1)\sin(l\theta+(n-l)\phi)\nonumber\\
&\quad+(n-1)\sin(n\theta)+\sin(n\phi)-n\sin\left((n-1)
\theta+\phi\right)\nonumber\\
&=2(n-1)\sin\left((n-1)\theta+\phi\right)+2\sin\left(\theta+(n-1)
\phi\right)-\sin(n\phi)-2\sin\left(\theta+(n-1)\phi\right)\nonumber\\
&\quad-(n-1)\sin(n\theta)-(n-2)\sin\left((n-1)\theta+\phi\right)
+(n-1)\sin(n\theta)+\sin(n\phi)-n\sin\left((n-1)\theta+
\phi\right)\nonumber\\
&=0, \qquad \forall\theta,\,\phi\in [-\pi,\pi],
\end{align}
which implies the trigonometric identity \eqref{sin}.
Similarly, we can prove the second trigonometric identity \eqref{cos}
and the details are omitted here for brevity.
\end{proof}

\fakesection{{\bf Appendix B}. A trigonometric inequality and its proof}
\begin{center}
{\bf Appendix B}. A trigonometric inequality and its proof
\end{center}
\setcounter{equation}{0}
\renewcommand{\theequation}{B.\arabic{equation}}
Here we prove a
trigonometric inequality which is  used to prove Theorem 4.3.

\begin{lem}\label{trignometric ineq coro}
 The following trigonometric inequality holds:
\begin{align}\label{trignometric ineq}
&a_n(\cos(n\theta)-\cos(n\phi)+n\sin(n\theta)\sin(\theta-
\phi))+b_n(\sin(n\theta)-\sin(n\phi)-n\cos(n\theta)
\sin(\theta-\phi))\nonumber\\
&\geq -(1-\cos(\theta-\phi))n^2\sqrt{a_n^2+b_n^2}, \qquad \forall\theta,\,\phi\in [-\pi,\pi], \qquad \forall n\in\mathbb{Z}^+.
\end{align}
\end{lem}

\begin{proof}
Using \eqref{sin} and \eqref{cos}, noticing that
\[
1-\cos(\theta-\phi)\geq 0, \qquad
a_n\cos n\theta+b_n\sin n\theta\geq -\sqrt{a_n^2+b_n^2}, \qquad \forall\theta,\,\phi\in [-\pi,\pi], \qquad \forall n\in\mathbb{Z}^+,
\]
we have
\begin{align}
&a_n(\cos(n\theta)-\cos(n\phi)+n\sin(n\theta)
\sin(\theta-\phi))+b_n(\sin(n\theta)-\sin(n\phi)-
n\cos(n\theta)\sin(\theta-\phi))\nonumber\\
&=a_n(1-\cos(\theta-\phi))\left(n\cos(n\theta)+
\sum_{k=1}^{n-1}2k\cos\left(k\theta+(n-k)\phi\right)\right)\nonumber\\
&\ +b_n(1-\cos(\theta-\phi))\left(n\sin(n\theta)+
\sum_{k=1}^{n-1}2k\sin\left(k\theta+(n-k)\phi\right)\right)\nonumber\\
&\geq -(1-\cos(\theta-\phi))\left(n+\sum_{k=1}^{n-1}2k\right)
\sqrt{a_n^2+b_n^2}\nonumber\\
&=(1-\cos(\theta-\phi))n^2\sqrt{a_n^2+b_n^2},\qquad \forall\theta,\,\phi\in [-\pi,\pi], \qquad \forall n\in\mathbb{Z}^+,
\end{align}
which implies the desired inequality \eqref{trignometric ineq}.
\end{proof}

\end{appendices}

\section*{Acknowledgement}

This work was supported by the Academic Research Fund of the Ministry of Education of Singapore grant No.~MOE2019-T2-1-063 (R-146-000-296-112).
Part of the work was done when the authors were visiting the Institute of Mathematical Science at the National University of Singapore in 2020.

\section*{References}

\end{document}